\documentclass[11pt,reqno]{amsart}
\usepackage{amssymb,mathrsfs,graphicx}
\usepackage{ifthen}
\usepackage{colortbl}
\definecolor{black}{rgb}{0.0, 0.0, 0.0}
\definecolor{red}{rgb}{1.0, 0.5, 0.5}
\provideboolean{shownotes} 
\setboolean{shownotes}{true} 
\newcommand{\margnote}[1]{
\ifthenelse{\boolean{shownotes}}%
{\marginpar{\raggedright\tiny\texttt{#1}}}%
{}%
}
\newcommand{\hole}[1]{
\ifthenelse{\boolean{shownotes}}%
{\begin{center} \fbox{ \rule {.25cm}{0cm} \rule[-.1cm]{0cm}{.4cm}
\parbox{.85\textwidth}{\begin{center} \texttt{#1}\end{center}} \rule
{.25cm}{0cm}}\end{center}} {} }

\def\d{\,\mathrm{d}}
\def\tot#1#2{\frac{\d #1}{\d #2}}

\def\eps{\varepsilon}

\title[]{Cucker-Smale model with normalized communication weights and time delay}

\author[Choi]{Young-Pil Choi}
\address[Young-Pil Choi]{\newline Fakult\"at f\"ur Mathematik
    \newline  Technische Universit\"at M\"unchen, Boltzmannstra{\ss}e 3, 85748, Garching bei M\"unchen, Germany}
\email{ychoi@ma.tum.de}

\author[Haskovec]{Jan Haskovec}
\address[Jan Haskovec]{\newline Computer, Electrical and Mathematical Sciences \& Engineering
    \newline King Abdullah University of Science and Technology, 23955 Thuwal, KSA}
\email{jan.haskovec@kaust.edu.sa}

\numberwithin{equation}{section}

\newtheorem{theorem}{Theorem}[section]
\newtheorem{lemma}{Lemma}[section]

\newtheorem{remark}{Remark}[section]
\newtheorem{definition}{Definition}[section]
\newtheorem{assumption}{Assumption}

\def\({\begin{eqnarray}}
\def\){\end{eqnarray}}
\def\[{\begin{eqnarray*}}
\def\]{\end{eqnarray*}}

\newcommand{\R}{\mathbb R}

\newcommand{\mc}{\mathcal C}

\newcommand{\bq}{\begin{equation}}
\newcommand{\eq}{\end{equation}}

\newcommand{\lt}{\left}
\newcommand{\rt}{\right}

\newcommand{\pa}{\partial}

\newcommand{\mt}{\mathcal{T}}
\newcommand{\ms}{\mathcal{S}}

\def\Lyap{\mathcal{L}}

\def\P{\mathcal{P}}
\def\N{\mathbb{N}}
\def\d{\mathrm{d}}
\def\bx{\mathbf{x}}
\def\bv{\mathbf{v}}

\begin{document}
\allowdisplaybreaks

\date{\today}

\subjclass[]{}


\begin{abstract} We study a Cucker-Smale-type system with time delay in which agents interact with each other through normalized communication weights.
We construct a Lyapunov functional for the system and provide sufficient conditions for asymptotic flocking, i.e., convergence to a common velocity vector.
We also carry out a rigorous limit passage to the mean-field limit of the particle system as the number of particles tends to infinity.
For the resulting Vlasov-type equation we prove the existence, stability and large-time behavior of measure-valued solutions.
This is, to our best knowledge, the first such result for a Vlasov-type equation with time delay.
We also present numerical simulations of the discrete system with few particles that provide
further insights into the flocking and oscillatory behaviors of the particle velocities depending on the size of the time delay.
\end{abstract}

\maketitle \centerline{\date}

\tableofcontents

%
%
%
%
\section{Introduction}
Collective coordinated motion of autonomous self-propelled
agents with self-organization into robust patterns
appears in many applications ranging from animal herding to the emergence 
of common languages in primitive societies~\cite{Sumpter}.
Apart from its biological and evolutionary relevance, collective phenomena
play a prominent role in many other scientific disciplines, such
as robotics, control theory, economics and social 
sciences~\cite{Carrillo-review, Vicsek-survey, Pareschi-Toscani-survey}.

The Cucker-Smale model was introduced and studied in the seminal 
papers~\cite{CS1, CS2}, originally as a model for language evolution. 
Later the interpretation as a model for flocking in animals 
(birds) prevailed. The model considers a finite number $N\in\N$ of autonomous agents
located in the physical space $\R^d$, $d\geq 1$.
The agents are described by their phase-space 
coordinates $(x_i(t), v_i(t))\in\R^{2d}$, $i=1,2,\dots,N$,
where $x_i(t)$ denotes the position and $v_i(t)$ the velocity of the $i$-th agent.
The agents are subject to the following collective dynamics,
\begin{align}
\begin{aligned} \label{CS}
   \tot{x_i(t)}{t} &= v_i(t),\quad i=1,\cdots, N, \qquad t > 0, \\
   \tot{v_i(t)}{t} &= \sum_{k=1}^N \varphi(|x_k-x_i|) (v_k-v_i).
\end{aligned}
\end{align}
The communication rate $\varphi$ introduced in~\cite{CS1, CS2} and considered in most of 
the subsequent papers is of the form
\( \label{def_varphi}
   \varphi(s) = \frac{\lambda}{(1+s^2)^\beta},
\)
with the constants $\lambda>0$ and $\beta\in\R$.
We introduce the spatial and, resp., velocity diameters as follows,
\(  \label{dXdV}
   d_X(t) := \max_{1 \leq i,j \leq N}|x_i(t) - x_j(t)| \quad \mbox{and} \quad d_V(t) := \max_{1 \leq i,j \leq N}|v_i(t) - v_j(t)|.
\)
In general, the term \emph{flocking} refers to the phenomenon where autonomous agents reach 
a consensus based on limited environmental information and simple rules.
We stick to the commonly accepted mathematical definition, introduced by
Cucker and Smale:

\begin{definition}[Asymptotic flocking]\label{def:flocking}
We say that the system with particle positions $x_i(t)$ and velocities $v_i(t)$,
$i=1,\dots,N$ and $t\geq 0$, exhibits \emph{asymptotic flocking}
if the spatial and velocity diameters satisfy
\(  \label{flocking}
   \sup_{t\geq 0} d_X(t) < \infty,\qquad \lim_{t\to\infty} d_V(t) = 0.
\)
\end{definition}
 
The Cucker-Smale model \eqref{CS}--\eqref{def_varphi} is a simple relaxation-type model that reveals 
a phase transition depending on the intensity of communication between agents. 
If $\beta \leq 1/2$, then the model exhibits the so-called 
\emph{unconditional flocking}, where \eqref{flocking} holds for every initial configuration.
On the other hand, with $\beta> 1/2$ the flocking is \emph{conditional},
i.e., the asymptotic behaviour of the system depends on the value of $\lambda$
and on the initial configuration. 
This result was first proved in \cite{CS1, CS2} using tools from graph theory
(spectral properties of graph Laplacian), and slightly later reproved 
in~\cite{Tadmor-Ha} by means of elementary calculus. Another proof has been 
provided in~\cite{Ha-Liu}, based on a bound
by a system of dissipative differential inequalities, and, finally, the 
proof of~\cite{CFRT} is based on bounding the maximal velocity.
A rigorous derivation of the mean-field limit of the Cucker-Smale model was carried out in \cite{Ha-Liu}.

Various modifications of the classical Cucker-Smale model have 
been considered. For instance, the case of singular communication rates 
$\varphi(s) = 1 / s^\beta$ was studied in~\cite{CCH, CCMP, Ha-Liu, Peszek}.
Motsch and Tadmor~\cite{Motsch-Tadmor} scaled the communication rates
in terms of the relative distance between the agents, so that their 
model does not involve any explicit dependence on the number of agents.
The dependence of the communication rate on the topological rather 
than metric distance between agents was introduced in \cite{Haskovec}.
In \cite{CCHS}, systems of particles interacting through cut-off communication weights, for instance, a vision cone, were considered.
The influence of additive noise in individual velocity measurements 
was studied in~\cite{Ha-Lee-Levy} and~\cite{TLY},
while stochastic flocking dynamics with multiplicative white noises were 
considered in~\cite{Ahn-Ha}. The kinetic Cucker-Smale and Mostch-Tadmor equations with noise were studied in \cite{Choi, DFT},
showing the existence of a unique  global classical solution near Maxwellians and the convergence to them.

We are only aware of two papers where delays in information processing were
considered: In~\cite{Liu-Wu} a sufficient flocking condition is derived
for the Motsch-Tadmor variant of the model with processing delay.
In~\cite{EHS} the Cucker-Smale model with noise and delay
is studied and a sufficient flocking condition is derived in terms of noise intensity and delay length.
Let us note that, to our best knowledge, no analytic results exist in the literature
for a Vlasov-type equation with delay derived as the mean field limit of a delayed Cucker-Smale system.
 
We refer to \cite{CCP, CHL} and references therein for recent surveys on the Cucker-Smale type flocking models and its variants.
For more general collective behavior models, including first- and second-order systems,
we refer to, e.g., \cite{CCH2, CCHS, CCR, Jabin} and references therein.

In this paper we study a Cucker-Smale-type flocking system
with a fixed communication time-delay $\tau > 0$.
In particular, we assume that the agents are subject to the following collective dynamics,
\begin{align}\label{main_eq}
\begin{aligned}
\tot{x_i(t)}{t} &= v_i(t),\qquad i=1,\cdots,N, \quad t >0,\\
\tot{v_i(t)}{t} & = \sum_{k=1}^N \phi_{ik}(x,\tau) (v_k(t - \tau) - v_i(t)),
\end{aligned}
\end{align}
where $\phi_{ik}$ are the normalized communication weights given by
\(\label{phi}
\phi_{ik}(x,\tau) = \left\{ \begin{array}{ll}
\displaystyle \frac{\psi(|x_k(t - \tau) - x_i(t)|)}{\sum_{k\neq i} \psi(|x_k(t - \tau) - x_i(t)|)} & \textrm{if $k \neq i$,}\\[4mm]
 0 & \textrm{if $k=i$,}
  \end{array} \right.
\)
and $\psi: [0,\infty) \to (0,\infty)$ is the \emph{influence function}.
We consider the system subject to the initial datum
\(\label{IC0}
   x_i(s) = x^0_i(s), \quad v_i(s) = v^0_i(s),\qquad i=1,\cdots,N, \quad s \in [-\tau,0],
\)
i.e., we prescribe the initial position and velocity trajectories $x^0_i, v^0_i\in\mathcal C([-\tau,0]; \R^d)$.
For physical reasons it may be required that
$$  x_i^0(s) = x^0_i(-\tau) + \int_{-\tau}^s v_i^0(\sigma) \d\sigma\quad\mbox{for } s \in (-\tau,0],$$
but we do not pose this particular restriction here.

\begin{assumption}\label{ass:psi}
Throughout this paper we assume that the influence function $\psi$ is bounded, positive,
nonincreasing and Lipschitz continuous on $[0,\infty)$, with $\psi(0) = 1$.
\end{assumption}

In the first part of the paper we shall derive a sufficient condition
for asymptotic flocking for the system \eqref{main_eq}--\eqref{IC0}. 
The condition relates the decay properties of the influence function $\psi$
with the delay length $\tau$ and the spatial and velocity diameters of the initial datum.
Our strategy is first to show a uniform bound on the velocity diameter of the solution.
Then, applying a growth estimate on the convex hull of a set of particles velocities,
we show that the solution of the system \eqref{main_eq} is dominated by
a time-delayed system of dissipative differential inequalities analogous to the one proposed in \cite{Ha-Liu}.
This finally leads to the sought-for asymptotic flocking estimate.
We note that our flocking estimate refines the results of the previous works \cite{Liu-Wu, Motsch-Tadmor}.

The second part of the paper is devoted to the study of the mean-field limit of the particle system \eqref{main_eq}--\eqref{IC0}.
Letting formally $N\to\infty$ leads to the Vlasov-type kinetic equation for the one-particle distribution $f_t=f_t(x,v)$, which is
a time-dependent probability measure on the phase space $\R^d\times\R^d \equiv \R^{2d}$,
\bq\label{kin_mt}
\begin{array}{ll}
\pa_t f_t + v \cdot \nabla_x f_t + \nabla_v \cdot (F[f_{t-\tau}]f_t) = 0, \quad (x,v) \in \R^d \times \R^d, \quad t > 0,\\[4mm]
F[f_{t-\tau}](x,v) := \displaystyle \frac{\int_{\R^{2d}} \psi(|x-y|)(w-v)f(y,w,t-\tau)\,\d y\d w}{\int_{\R^{2d}} \psi(|x-y|)f(y,w,t-\tau)\,\d y\d w},\\[4mm]
f_s(x,v) = g_{s}(x,v), \quad (x,v) \in \R^d \times \R^d, \quad s \in [-\tau,0].
\end{array}
\eq
The initial datum $g_s$ is a time-dependent probability measure on the phase space $\R^{2d}$.
We shall prove the global existence and uniqueness of measure-valued solutions of \eqref{kin_mt}.
Moreover, we shall provide a stability estimate in terms of the Monge-Kantorowich-Rubinstein distance,
and, as a direct consequence, a bound on the error between the solutions of the kinetic equation \eqref{kin_mt}
and the empirical measure associated to the particle system \eqref{main_eq}--\eqref{IC0}.
Moreover, noting that the flocking estimate for the particle system does not depend on the number of particles enables us to
prove an asymptotic flocking result for the kinetic system.

The rest of this paper is organized as follows. In Section \ref{sec_par} we present our main result
on the flocking behavior of the discrete system \eqref{main_eq}--\eqref{IC0} and its proof.
Section \ref{sec_kin} is devoted to the rigorous derivation of the Vlasov-type equation \eqref{kin_mt}
from the discrete particle system \eqref{main_eq}--\eqref{IC0} in the mean-field limit as the number of particles $N$ goes to infinity.
Finally, in Section \ref{sec_nur} we provide results of numerical simulations of the discrete system with $N=2, 3, 4$ particles
and illustrate how the time-evolutions of particle velocities depend on the size of the time delay.

%
\section{Asymptotic behavior of the discrete particle model}\label{sec_par}

\begin{remark}[Existence of local solutions]\label{rem:existence}
By Assumption \ref{ass:psi}, the right-hand side of \eqref{main_eq} is locally Lipschitz continous
as a function of $(x_i(t), v_i(t))$, and thus, due to the Cauchy-Lipschitz theorem,
the ODE system \eqref{main_eq}--\eqref{IC0} admits a unique local-in-time $\mc^1$-solution.
We will denote it $\bx = (x_1,\dots,x_N)$ and $\bv = (v_1,\dots,v_N)$ in the sequel.
\end{remark}

The main result of this section is the following Theorem establishing
the global existence of solutions of the system \eqref{main_eq}--\eqref{IC0}
and providing a sufficient condition for asymptotic flocking as per Definition \ref{def:flocking}.

\begin{theorem}\label{thm_main}
Suppose that the initial velocity profiles $v_i^0$ are continuous and bounded on the time interval $[-\tau,0]$
and denote
\(   \label{Rv}
    R_v := \max_{s \in [-\tau,0]} \max_{1 \leq i \leq N} |v^0_i(s)|.
\)
Moreover, assume that 
\(  \label{ass1}
  d_V(0)  + \int_{-\tau}^0 d_V(s)\,ds < \int_{d_X(-\tau) + R_v \tau}^\infty \psi(s)\,ds,
\)
where $d_X$ and, resp., $d_V$ denote the spatial and, resp., velocity diameters defined in \eqref{dXdV}.
Then the solution $(\bx, \bv)$ of the system \eqref{main_eq}--\eqref{IC0} is global in time
and satisfies
\[
   d_V(t) \leq \left( \max_{s\in[-\tau,0]} d_V(s) \right) e^{-Ct} \quad \mbox{for } t \geq 0,
   \qquad \sup_{t\geq 0} d_X(t) < \infty,
\]
where $C$ is a positive constant independent of $t$ and $N$.
\end{theorem}

\begin{remark}
The assumption \eqref{ass1} can be understood, for a fixed integrable influence function $\psi$,
as a condition for smallness of the delay $\tau$. Indeed, considering a fixed initial datum with
$d_V(s)\equiv: d_V^0>0$ constant for $s\in [-\tau,0]$ and $d_X(-\tau)\equiv: d_X^0\geq 0$ for all $\tau>0$ , then \eqref{ass1} reads
\[
   (1+\tau) d_V^0 < \int_{d_X^0 + R_v \tau}^\infty \psi(s)\,ds.
\]
Clearly, the left-hand side increases with increasing $\tau$, while the right-hand side decreases.
So, generically, it is necessary to choose $\tau$ sufficiently small in order to satisfy the flocking condition.
This is often the case in alignment models with delay, see, e.g., \cite{EHS}.

On the other hand, if the influence function $\psi$ has a heavy tail, i.e.,
\[
   \int^\infty \psi(s)\,d s = \infty,
\]
then assumption \eqref{ass1} is satisfied for any initial datum and any $\tau \geq 0$,
which is a situation usually called \emph{unconditional flocking}, see, e.g., \cite{CS1, CS2} and \cite{Liu-Wu}.
\end{remark}

\begin{remark}\label{rem:no_delay}
If there is no time delay, i.e., $\tau = 0$,
then our system \eqref{main_eq}--\eqref{phi} becomes identical to the one studied in \cite{Motsch-Tadmor}.
There, the unconditional flocking estimate is obtained under the assumption $\int^\infty \psi(s)^2\,ds = \infty$.
Thus, our result in Theorem \ref{thm_main} can be seen as a refinement of the unconditional flocking condition of \cite{Motsch-Tadmor}. 
\end{remark}

For the proof of Theorem \ref{thm_main} we will need several auxiliary results.

\begin{lemma}\label{lem_bdd}
Let $R_v > 0$ be given by \eqref{Rv} and
let $(\bx,\bv)$ be a local-in-time $\mathcal{C}^1$-solution of the system \eqref{main_eq}--\eqref{IC0}
constructed in Remark \ref{rem:existence}.
Then the solution is global in time and satisfies
\[
   \max_{1 \leq i \leq N} |v_i(t)| \leq R_v \quad \mbox{for} \quad t \geq -\tau.
\]
\end{lemma}

\begin{proof}
We first notice that the equations $\eqref{main_eq}_2$ can be rewritten as
\bq\label{est_bdd}
   \tot{v_i(t)}{t} = \sum_{k=1}^N \phi_{ik}(x,\tau)v_k(t - \tau) - v_i(t),
\eq
due to the identity
\bq\label{eq_nor}
   \sum_{k=1}^N \phi_{ik}(x,\tau) = 1 \qquad\mbox{for } i = 1,\dots, N.
\eq
Choose any $\eps>0$, set $R_v^\eps := R_v + \eps$ and
\[
   \ms^\eps := \lt\{ t>0 :\, \max_{1 \leq i \leq N}|v_i(s)| < R_v^\eps \quad \mbox{for} \quad s \in [0,t) \rt\}.
\]
By the assumption, we have $\ms^\eps \neq \emptyset$.
Set $T_*^\eps := \sup \ms^\eps > 0$. We then claim $T_*^\eps = \infty$.
For contradiction, suppose $T_*^\eps < \infty$. This yields
\bq\label{est_claim}
   \lim_{t \to T_*^\eps-} \; \max_{1 \leq i \leq N}|v_i(t)| = R_v^\eps.
\eq
On the other hand, from \eqref{est_bdd} and \eqref{eq_nor} it follows that for $t  < T_*^\eps$ and any $i=1,\dots,N$,
$$\begin{aligned}
\frac12\tot{|v_i(t)|^2}{t} &\leq \sum_{k=1}^N \phi_{ik}(x,\tau)|v_k(t - \tau)||v_i(t)| - |v_i(t)|^2\cr
&\leq \max_{1 \leq k \leq N}|v_k(t - \tau)||v_i(t)| - |v_i(t)|^2\cr
&\leq R_v^\eps|v_i(t)| - |v_i(t)|^2.
\end{aligned}$$
Now, if $|v_i(t)| \neq 0$, we use the identity $\frac12\tot{|v_i(t)|^2}{t} = |v_i(t)| \tot{|v_i(t)|}{t}$
and we may divide the above inequality by $|v_i(t)|$.
On the other hand, if $|v_i(t)| \equiv 0$ on an open subinterval of $[0,T_*^\eps)$,
then $\tot{|v_i(t)|}{t} \equiv 0 \leq R_v^\eps  - |v_i(t)|$ on this subinterval.
Thus, we obtain
\[
   \tot{|v_i(t)|}{t} \leq R_v^\eps - |v_i(t)| \quad\mbox{a.e. on } (0,T_*^\eps), \qquad |v_i(0)| < R_v^\eps,
\]
which implies, due to the continuity of $|v_i(t)|$ on $[0,T_*^\eps)$,
\[
   |v_i(t)| \leq \lt(|v_i(0)| - R_v^\eps \rt)e^{-t} + R_v^\eps  \qquad\mbox{for } t  < T_*^\eps.
\]
Consequently,
\[
   \lim_{t \to T_*^\eps-} \; \max_{1 \leq i \leq N} |v_i(t)| \leq \lt(\max_{1 \leq i \leq N}|v_i(0)| - R_v^\eps \rt)e^{-T_*} + R_v^\eps < R_v^\eps,
\]
due to $\max_{1 \leq i \leq N}|v_i(0)| < R_v^\eps$. This is a contradiction to \eqref{est_claim} and we conclude that $T_*^\eps = \infty$.
Finally, we complete the proof by taking the limit $\eps\to 0$.
\end{proof}

In the sequel we will need the following auxiliary lemma:

\begin{lemma}\label{prop:aux}
Let $v_1, \dots, v_N\in\R^d$ be any set of vectors and denote $d_V:=\max_{1 \leq i,j \leq N}|v_i - v_j|$.
Fix $0 < \kappa \leq 1/N$ and set
\[
   \Omega_\kappa := \left\{ \sum_{i=1}^N \psi_i v_i; \mbox{ with } \psi_i\geq\kappa,\, i=1,\dots,N,\; \sum_{i=1}^N \psi_i = 1 \right\}.
\]
Then $\mathrm{diam}(\Omega_\kappa) \leq (1-\kappa N)d_V$.
\end{lemma}

\begin{proof} 
For $i, j = 1,\dots, N$ set $\xi_{ij}:=\kappa$ if $i\neq j$ and $\xi_{ii}:=1 - (N-1)\kappa$.
Moreover, define
\[
   \bar v_i := \sum_{j=1}^N \xi_{ij} v_j \qquad\mbox{for } i = 1,\dots, N.
\]
We prove that the set $\Omega_\kappa$ equals to the convex hull of the vectors $\bar v_1, \dots, \bar v_N$.
Indeed, let $w\in\R^d$ be a convex combination of $\bar v_1, \dots, \bar v_N$, i.e., there exist scalars
$0 \leq \lambda_i \leq 1$ such that $\sum_{i=1}^N \lambda_i = 1$ and
\[
   w = \sum_{i=1}^N \lambda_i \bar v_i = \sum_{i=1}^N \sum_{j=1}^N \lambda_i \xi_{ij} v_j.
\]
Defining $\psi_j := \sum_{i=1}^N \lambda_i \xi_{ij}$, it is easy to prove that
$\kappa\leq \psi_j$ and $\sum_{j=1}^N \psi_j = 1$.
Consequently, $w\in\Omega_\kappa$.

On the other hand, if $w\in\Omega_\kappa$, i.e.,
\[
   w = \sum_{i=1}^N \psi_i v_i \qquad\mbox{with } \psi_i\geq\kappa \mbox{ for } i=1,\dots,N,\; \sum_{i=1}^N \psi_i = 1,
\]
then it can be easily checked that $w$ can be written as the convex combination
\[
   w = \sum_{i=1}^N \lambda_i \bar v_i \qquad\mbox{with } \lambda_i = \frac{\psi_i-\kappa}{1-\kappa N} \geq 0,\;
   \sum_{i=1}^N \lambda_i = 1.
\]

Finally, a direct calculation gives
\[
   |\bar v_i - \bar v_j| \leq (1-\kappa N) |v_i - v_j| \qquad\mbox{for any } i,j=1,\dots,N,
\]
so that
\[
   \mathrm{diam}(\Omega_\kappa) = \max_{1 \leq i,j \leq N}|\bar v_i - \bar v_j| \leq (1-\kappa N) \max_{1 \leq i,j \leq N}|v_i - v_j|
    = (1-\kappa N) d_V.   
\]
\end{proof}

In the lemma below we derive the differential inequalities for the spatial and, resp., velocity diameters
$d_X$ and, resp., $d_V$, which will be instrumental for proving Theorem \ref{thm_main}.

\begin{lemma}\label{lem_sddi}
Let $R_v > 0$ be given by \eqref{Rv} and
let $(\bx,\bv)$ be the global $\mathcal{C}^1$-solution of the system \eqref{main_eq}--\eqref{IC0}
constructed in Lemma \ref{lem_bdd}.
Then for almost all $t>0$ we have
$$\begin{aligned}
\tot{}{t}d_X(t) &\leq d_V(t), \cr
\tot{}{t} d_V(t) &\leq (1 - \psi(d_X(t) + R_v \tau))d_V(t - \tau) - d_V(t).
\end{aligned}$$
\end{lemma}

\begin{proof}
Due to the continuity of the velocity trajectories $v_i(t)$,
there is an at most countable system of open, mutually disjoint
intervals $\{\mathcal{I}_\sigma\}_{\sigma\in\N}$ such that
$$
   \bigcup_{\sigma\in\N} \overline{\mathcal{I}_\sigma} = [0,\infty)
$$
and for each ${\sigma\in\N}$ there exist indices $i(\sigma)$, $j(\sigma)$
such that
$$
   d_V(t) = |v_{i(\sigma)}(t) - v_{j(\sigma)}(t)| \quad\mbox{for } t\in \mathcal{I}_\sigma.
$$
Then, using the abbreviated notation $i:=i(\sigma)$, $j:=j(\sigma)$,
we have for every $t\in \mathcal{I}_\sigma$,
$$
\begin{aligned}
\frac12\tot{}{t} d_V(t)^2 &= (v_i(t) - v_j(t)) \cdot \lt(\tot{v_i(t)}{t} - \tot{v_j(t)}{t}\rt)\cr
&= (v_i(t) - v_j(t)) \cdot \lt(\sum_{k=1}^N \phi_{ik}(x,\tau)v_k(t - \tau) -  \sum_{k=1}^N \phi_{jk}(x,\tau)v_k(t - \tau)\rt)\cr
&\quad  - |v_i(t) - v_j(t)|^2.
\end{aligned}
$$
Using $\eqref{main_eq}_1$, we estimate for any $1 \leq i, k \leq N$,
$$\begin{aligned}
|x_k(t - \tau) - x_i(t)| &= \lt|x_k(t-\tau) - x_i(t-\tau) -\int^{t}_{t-\tau} \tot{}{t}{x}_i(s)\,ds\rt| \cr
&\leq |x_k(t-\tau) - x_i(t-\tau)| + \tau \sup_{s \in (t-\tau, t)}|v_i(s)|.
\end{aligned}$$
Then, Lemma \ref{lem_bdd} gives
\[
   |x_k(t - \tau) - x_i(t)| \leq d_X(t-\tau) + R_v \tau,
\]
and due to the monotonicity properties of the influence function $\psi$ (Assumption \ref{ass:psi}),
\[
  \phi_{ik}(x,\tau) \geq \frac{\psi(d_X(t-\tau) + R_v \tau)}{N}. 
\]
Thus, denoting $\kappa:=\psi(d_X(t-\tau) + R_v \tau)/N$, we have $\phi_{ik}(x,\tau) \geq\kappa$ and $\sum_{k=1}^N \phi_{ik}(x,\tau) = 1$.
An application of Lemma \ref{prop:aux} implies then
\[
   \lt|\sum_{k=1}^N \phi_{ik}(x,\tau)v_k(t - \tau) -  \sum_{k=1}^N \phi_{jk}(x,\tau)v_k(t - \tau)\rt| \leq (1 - \kappa N) d_V(t-\tau),
\]
so that, for almost all $t>0$,
\[
   \frac12\tot{}{t} d_V(t)^2 \leq (1-\psi(d_X(t-\tau) + R_v \tau)) d_V(t-\tau) d_V(t) - d_V(t)^2.
\]
Now, if $d_V(t) \neq 0$, we use the identity $\frac12\tot{}{t} d_V(t)^2 = d_V(t) \tot{}{t} d_V(t)$
and we may divide the above inequality by $d_V(t)$.
On the other hand, if $d_V(t) \equiv 0$ on an open subinterval of $[0,\infty)$,
then $\tot{}{t} d_V(t) \equiv 0 \leq \lt(1 - \psi(d_X(t-\tau) + R_v \tau)\rt)d_V(t - \tau) - d_V(t)$ on this subinterval.
Consequently, we have
\[
   \tot{}{t}d_V(t) \leq \lt(1 - \psi(d_X(t-\tau) + R_v \tau)\rt)d_V(t - \tau) - d_V(t)
\]
for almost all $t>0$.
\end{proof}

Finally, we prove the following generalized Gronwall-type inequality.

\begin{lemma}\label{lem_gron}
Let $u$ be a nonnegative, continuous and piecewise $\mc^1$-function satisfying,
for some constant $0 < a < 1$, the differential inequality
\bq\label{diff_ineq}
   \tot{}{t}  u(t) \leq (1-a) u(t -\tau) - u(t) \qquad\mbox{for almost all } t>0.
\eq
Then there exists a constant $0 < C < 1$ satisfying the equation
\bq\label{ass_C}
  1 - C = (1-a)e^{C\tau}
\eq
and such that the estimate holds
\( \label{Gronwall-like}
   u(t) \leq \left( \max_{s\in[-\tau,0]} u(s) \right) e^{-Ct} \qquad \mbox{for all } t \geq 0.
\)
\end{lemma}

\begin{proof}
Denote
\[
    \bar u := \max_{s\in[-\tau,0]} u(s), \qquad w(t) := \bar u e^{-Ct},
\]
and for any $\lambda > 1$ set
\[
   \ms^\lambda := \lt\{ t \geq 0 : u(s) \leq \lambda w(s) \quad \mbox{for} \quad s \in [0,t)\rt\}.
\]
Since $0 \in \ms^\lambda$, $T^\lambda := \sup \ms^\lambda$ exists. We claim that
\[
   T^\lambda = \infty \qquad \mbox{for any } \lambda >0.
\]
For contradiction, assume $T^\lambda < \infty$ for some $\lambda >1$.
Then, due to the continuity of $u$, there exists a $T^\lambda_* \in (T^\lambda, T^\lambda+\tau)$ such that
$u$ is differentiable at $T^\lambda_*$ and
\(  \label{est_derivatives}
   u(T^\lambda_*) > \lambda w(T^\lambda_*),\qquad   \tot{}{t} u(T^\lambda_*) \geq \lambda \tot{}{t} w(T^\lambda_*).
\)
Note that $w$ satisfies
\bq\label{est_g}
   w(t-\tau) = e^{C\tau}w(t) \quad \mbox{and} \quad \tot{}{t} w(t) = - C w(t).
\eq
It follows from \eqref{diff_ineq} that
\bq\label{est_diff}
   \tot{}{t} u(T^\lambda_*) \leq (1 - a) u(T^\lambda_* - \tau) - u(T^\lambda_*).
\eq
We now consider the following two cases:
\begin{itemize}
\item
If $0 \leq T^\lambda_* \leq \tau$, then, by definition, $u(T^\lambda_* - \tau) \leq \bar u$,
and we estimate the right-hand side of \eqref{est_diff} by
$$
\begin{aligned}
    \tot{}{t} u(T^\lambda_*)  &\leq (1-a) \bar u - u(T^\lambda_*)\cr
   &< (1-a) \lambda w(0) - \lambda w(T^\lambda_*)\cr
   &\leq \lt((1 - a)e^{C\tau} - 1\rt) \lambda w(T^\lambda_*),
\end{aligned}
$$
where used the inequality $w(0) \leq w(T^\lambda_*)e^{C\tau}$.
With the identities \eqref{ass_C} and \eqref{est_g}, we obtain
$$
   \tot{}{t} u(T^\lambda_*)  < - C \lambda w(T^\lambda_*)  = \lambda \tot{}{t} w(T^\lambda_*),
$$
which is a contradiction to \eqref{est_derivatives}.

\item
If $T^\lambda_* > \tau$, we have $u(T^\lambda_* - \tau) \leq \lambda w(T^\lambda_* - \tau)$
by construction of $T^\lambda_*$, and \eqref{est_diff} gives
\[
   \tot{}{t} u(T^\lambda_*) &<& (1-a) \lambda w(T^\lambda_* - \tau) - \lambda w(T^\lambda_*)\\
     &=& \lt((1 - a)e^{C\tau} - 1\rt) \lambda w(T^\lambda_*).
\]
Then, using the same argument as in the previous case, we obtain a contradiction to \eqref{est_derivatives}.
\end{itemize}
We conclude that, for every $\lambda>1$, $T^\lambda = \infty$ and $u(t) \leq \lambda w(t)$ for all $t\geq 0$.
Passing to the limit $\lambda \to 1$ yields the claim \eqref{Gronwall-like}.
\end{proof}

We are now ready to proceed with the proof of Theorem \ref{thm_main}.

\begin{proof}[Proof of Theorem \ref{thm_main}]
We introduce for $t > 0$ the following Lyapunov functional for the system \eqref{main_eq}--\eqref{IC0},
\[
   \Lyap(t) := d_V(t) + \int_{d_X(- \tau) + R_v \tau}^{d_X(t - \tau) + R_v \tau} \psi(s)\,\d s + \int_{-\tau}^0 d_V(t + s)\,\d s,
\]
where $R_v$ is given by \eqref{Rv} and the diameters $d_X(t)$, $d_V(t)$ by \eqref{dXdV}.
Then using Lemma \ref{lem_sddi}, we calculate, for almost all $t>0$,
$$
\begin{aligned}
   \tot{}{t}\Lyap(t) &= \tot{}{t} d_V(t) + \psi(d_X(t - \tau) + R_v \tau) \tot{}{t} d_X(t - \tau) + d_V(t) - d_V(t - \tau)\cr
   &\leq \lt(1 - \psi(d_X(t-\tau) + R_v \tau)\rt)d_V(t - \tau) - d_V(t) \cr 
   &\quad + \psi(d_X(t - \tau) + R_v \tau) d_V(t-  \tau) +d_V(t) - d_V(t - \tau)\cr
   &= 0.
\end{aligned}
$$
Thus, integrating over the time interval $(0,t)$, we obtain
\(  \label{zwischenstep}
   d_V(t) + \int_{d_X(- \tau) + R_v \tau}^{d_X(t - \tau) + R_v \tau} \psi(s)\,\d s + \int_{-\tau}^0 d_V(t + s)\,\d s \leq d_V(0)  + \int_{-\tau}^0 d_V(s)\,\d s.
\)
From assumption \eqref{ass1} it follows that there exists a $d_* > 0$ such that
\[
   d_V(0)  + \int_{-\tau}^0 d_V(s)\,\d s = \int_{d_X(-\tau) + R_v \tau}^{d_*}\psi(s)\,\d s.
\]
Combining with \eqref{zwischenstep}, we get
\[
   \int_{d_X(- \tau) + R_v \tau}^{d_X(t - \tau) + R_v \tau} \psi(s)\,\d s \leq \int_{d_X(-\tau) + R_v \tau}^{d_*}\psi(s)\,\d s,
\]
which implies
\[
   0  \leq \int_{d_X(t -\tau) + R_v \tau}^{d_*} \psi(s)\,\d s,
\]
so that
\[
   d_X(t -\tau) + R_v \tau \leq d_* \quad \mbox{for }  t > 0.
\]
Hence, by Lemma \ref{lem_sddi} and the monotonicity of $\psi$ we have, for almost all $t>0$,
\[
   \tot{}{t} d_V(t) \leq (1 - \psi_*)d_V(t - \tau) - d_V(t),
\]
where $\psi_* = \psi(d_*)$. We finally apply Lemma \ref{lem_gron} to complete the proof.
\end{proof}

%
\section{Measure-valued solutions: existence, stability, and large-time behavior }\label{sec_kin}
In this section we provide a proof of global existence of measure-valued solutions for the kinetic model \eqref{kin_mt}
and a stability estimate in Monge-Kantorowich-Rubinstein distance.
The stability estimate allows us to show rigorously that \eqref{kin_mt} is the mean-field limit
of the discrete system \eqref{main_eq}--\eqref{phi}.
Moreover, it allows to obtain a sufficient condition for flocking in the kinetic system,
analogous to \eqref{ass1}.

Let $\P_1(\R^{2d})$ be the set of probability measures on the phase space $\R^{2d}$
with bounded first-order moment.
The Monge-Kantorovich-Rubinstein distance, also called 1-Wasserstein distance,
is defined as follows.

\begin{definition}(Monge-Kantorovich-Rubinstein distance) \label{defdp}
Let $\rho^1, \rho^2\in \P_1(\R^{d})$ be two probability measures on $\R^{d}$. 
Then the Monge-Kantorovich-Rubinstein distance between $\rho^1$ and $\rho^2$ is defined as
\begin{equation*}\label{d1}
d_1(\rho^1,\rho^2) := \inf_{\gamma\in\Gamma(\rho^1,\rho^2)} \int_{\R^{2d}} |x-y|  \d\gamma(x,y) ,
\end{equation*}
where $\gamma\in\Gamma(\rho^1,\rho^2)$ is the set of transference plans, i.e.,
probability measures $\gamma$ on $\R^{d} \times \R^{d}$ with
marginals $\rho^1$ and $\rho^2$,
\[
   \int_{\R^{2d}} \xi(x) d\gamma(x,y) = \int_{\R^{d}} \xi(x) \d\rho^1(x),
\]
and
\[
   \int_{\R^{2d}} \xi(y) d\gamma(x,y) = \int_{\R^{d}} \xi(y) \d\rho^2(y),
\]
for all continuous and bounded functions $\xi \in \mathcal{C}_b(\R^{d})$.
\end{definition}

Note that $\P_1(\R^{2d})$ endowed with the Monge-Kantorovich-Rubinstein distance is a complete metric space.
Moreover, the Monge-Kantorovich-Rubinstein distance is equivalent to the Bounded Lipschitz distance,
i.e., for any $\rho^1,~ \rho^2\in \P_1(\R^{d})$,
\[\label{blipd}
   d_1(\rho^1,\rho^2) = \sup\left\{ \lt|\int_{\R^{d}} \varphi(\xi)\d\rho^1(\xi) -  \int_{\R^{d}} \varphi(\xi)\d\rho^2(\xi)\rt| ; \varphi \in \mbox{Lip}(\R^d), \mbox{ Lip}(\varphi) \leq 1\right\},
\]
where Lip($\R^{d}$) denotes the set of Lipschitz functions on $\R^{d}$ and Lip($\varphi$) is the Lipschitz constant of the function $\varphi \in$ Lip($\R^{d}$).

We also recall the definition of the push-forward of a measure by a mapping:
\begin{definition}\label{def:push-forward}
Let $\rho$ be a Borel measure on $\R^{d}$ and $\mathcal{T} :
\R^{d} \to \R^{d}$ be a measurable mapping. Then the push-forward
of $\rho$ by $\mathcal{T}$ is the measure $\mathcal{T} \# \rho$ defined by
\[
\mathcal{T} \# \rho(B) := \rho(\mathcal{T}^{-1}(B)) \qquad \mbox{for all Borel sets }  B\subset \R^d.
\]
\end{definition}

We next define the notion of measure-valued solutions to the equation \eqref{kin_mt}.

\begin{definition}\label{def_weak}
For a given $T >0$, we call $f_t \in \mc([0,T); \P_1(\R^{2d}))$ a \emph{measure-valued solution} of the equation \eqref{kin_mt}
on the time interval $[0,T)$, subject to the initial datum $g_s\in \mc([-\tau,0]; \P_1(\R^{2d}))$,
if for all compactly supported test functions $\xi \in \mc_c^\infty(\R^{2d} \times [0,T))$,
\bq\label{weak_for}
\int_0^T \int_{\R^{2d}} f_t\lt(\pa_t \xi + v \cdot \nabla_x \xi + F[f_{t-\tau}]\cdot\nabla_v\xi\rt)\,\d x\d v\d t
+ \int_{\R^{2d}} g_0(x,v)\xi(x,v,0)\,\d x\d v = 0,
\eq
where $F[f_{t-\tau}]$ is defined in \eqref{kin_mt}
and we adopt the notation $f_{t-\tau}:\equiv g_{t-\tau}$ for $t \in [0,\tau]$.
\end{definition}

Note that in the case of asymptotic flocking the distribution function $f=f(t,x,v)$ will concentrate
in the velocity variable as $t\to\infty$, hence the space of nonnegative (probability) measures
is the natural solution space for the kinetic model.


%
\subsection{Existence and uniqueness of solutions for the kinetic problem \eqref{kin_mt}}
In this part, we show the existence and uniqueness of measure-valued solutions to the kinetic equation \eqref{kin_mt} in the sense of Definition \ref{def_weak}.

The main result of this section is:
\begin{theorem}\label{main_thm2}
Let the initial datum $g_t \in  \mc([-\tau,0]; \P_1(\R^{2d}))$
and assume that $g_t$ is uniformly compactly supported in position and velocity on $[-\tau,0]$, i.e.,
there exists a constant $R > 0$ such that
\bq\label{kinetic_supp}
  \mbox{supp } g_t \subset B^{2d}(0,R) \qquad\mbox{for all } t \in [-\tau,0],
\eq
where $B^{2d}(0,R)$ denotes the ball of radius $R$ in $\R^{2d}$, centered at the origin.

Then for any $T > 0$, the kinetic equation \eqref{kin_mt} admits a unique measure-valued solution $f_t\in  \mc([0,T); \P_1(\R^{2d}))$
in the sense of Definition of \eqref{def_weak}, which is also uniformly compactly supported in position and velocity.
Furthermore, $f_t$ is determined as the push-forward of the density $f_0 := g_0$ through the flow map
generated by the locally Lipschitz velocity field $(v,F[f_{t-\tau}])$ in phase space.
\end{theorem}

For the proof, we first derive a local Lipschitz and $L^\infty$ bound on the force field $F[f_t]$.
Let us introduce the zeroth and first order
moments of $f_t$,
\[
   \rho_{f_t} := \int_{\R^d} \d f_t(v) \quad \mbox{and} \quad \rho_{f_t} u_{f_t} := \int_{\R^d} v \,\d f_t (v).
\]
Note that $\rho_{f_t}$ is, for every fixed $t\geq 0$, a probability measure on $\R^d$.

\begin{lemma}\label{lem:LipschF}
Let $f_t \in \mc([0,T];\P_1(\R^{2d}))$ be uniformly compactly supported, i.e.,
\[
  \mbox{supp } f_t \subset B^{2d}(0,R) \qquad\mbox{for all } t \in [0,T],
\]
for some positive constant $R > 0$.

Then the force field $F[f_t](x,v)$ defined in \eqref{kin_mt} is locally Lipschitz continuous with respect to $x$ and $v$,
uniformly in time, i.e.,
there exists a constant $C > 0$ such that 
\[
|F[f_t](x,v) - F[f_t](\tilde x,\tilde v)| \leq C\bigl( |x-\tilde x| + |v - \tilde v|\bigr) \quad \mbox{for} \quad (x,v) \in B^{2d}(0,R),\; t\in [0,T].
\]
Moreover, there exists a constant $C > 0$ such that
\[
   |F[f_t](x,v)| \leq C \quad \mbox{for} \quad (x,v) \in B^{2d}(0,R),\; t\in [0,T].
\]
\end{lemma}

\begin{proof}
Due to the assumed positivity of the influence function $\psi$ and the fact that $\rho_{f_t}$ is a probability measure on $\R^d$,
we have
\[
   0 <  \psi_* := \inf_{y \in B^d(0,2R)} \psi(|y|) \leq (\psi \star \rho_{f_t})(x)
   \qquad\mbox{for all } x\in\R^d,
\]
where the convolution operator $\star$ is defined as
\[
    (\psi \star \rho)(x) :=  \int_{\R^{d}} \psi(|x-y|) \d\rho(y).
\]
Moreover,
\[
   \lt|(\psi \star (\rho_{f_t} u_{f_t}))(x)\rt| = \lt|\int_{\R^{2d}} \psi(|x-y|) v\,\d f_t(y,v)\rt| \leq R\|\psi\|_{L^\infty}.
\]
This yields
$$\begin{aligned}
&|F[f_t](x,v) - F[f_t](\tilde x, \tilde v)| \cr
&\,\, \leq \lt|\frac{(\psi \star (\rho_{f_t} u_{f_t}))(x)}{(\psi \star \rho_{f_t})(x)} - \frac{(\psi \star (\rho_{f_t} u_{f_t}))(\tilde x)}{(\psi \star \rho_{f_t})(\tilde x)} \rt| + |v - \tilde v|\cr
&\,\, \leq \frac{2\max\{1, R\}\|\psi\|_{L^\infty}}{\psi_*^2}\lt(\lt|(\psi \star (\rho_{f_t} u_{f_t}))(x) - (\psi \star (\rho_{f_t} u_{f_t}))(\tilde x) \rt| + \lt|(\psi \star\rho_{f_t})(x) - (\psi \star \rho_{f_t})(\tilde x) \rt|\rt)\cr
&\quad + |v - \tilde v|\cr
&\,\, \leq \frac{2(R+1)\max\{1, R\}\|\psi\|_{L^\infty}\|\psi\|_{Lip}}{\psi_*^2}|x - \tilde x| + |v - \tilde v|,
\end{aligned}$$
where $\|\psi\|_{Lip}$ denotes the Lipschitz constant of $\psi$ and we used the estimates
\[
\lt|(\psi \star (\rho_{f_t} u_{f_t}))(x) - (\psi \star (\rho_{f_t} u_{f_t}))(\tilde x) \rt| \leq R\|\psi\|_{Lip}|x - \tilde x|
\]
and
\[
\lt|(\psi \star\rho_{f_t})(x) - (\psi \star \rho_{f_t})(\tilde x) \rt| \leq \|\psi\|_{Lip}|x - \tilde x|.
\]
Finally, we easily find that
\[
   |F[f_t](x,v)| \leq R\lt(\frac{\|\psi\|_{L^\infty}}{\psi_*} + 1\rt) \quad \mbox{for} \quad (x,v) \in B^{2d}(0,R),\; t\in [0,T].
\]
\end{proof}

We are now ready to prove the main Theorem of this Section.

\begin{proof}[Proof of Theorem \ref{main_thm2}]
An application of \cite[Theorem 3.10]{CCR} together with Lemma \ref{lem:LipschF}
directly implies the local-in-time existence and uniqueness of measure-valued solutions
to the system \eqref{kin_mt} in the sense of Definition \ref{def_weak}.
We notice that these local-in-time solutions exist as long as the solution is compactly supported in position and velocity.
Thus, to prove the global-in-time existence of solutions, we only need to estimate the growth of support of $f_t$ in both position and velocity.
Let us set 
\(   \label{RXRV}
   R_X[f_t] := \max_{x \,\in\, \overline{\mbox{\footnotesize supp}_x f_t}}|x|,  \qquad
   R_V[f_t] := \max_{v \,\in\, \overline{\mbox{\footnotesize supp}_v f_t}}|v|,
\)
   for $t \in [0,T]$, where supp$_x f_t$ and supp$_v f_t$ represent $x$- and $v$-projections of supp$f_t$, respectively. We also set
\(  \label{supp_diams}
   R^t_X := \max_{-\tau \leq s \leq t}R_X[f_s], \qquad  R^t_V := \max_{-\tau \leq s \leq t}R_V[f_s].
\)
We will construct the solution using the method of steps, see, e.g., \cite{Smith}.
We first consider the time interval $[0,\tau]$ and construct the system of characteristics
$Z(t;x,v) := (X(t;x,v),V(t;x,v)): [0,\tau] \times \R^d \times \R^d \to \R^d \times \R^d$
associated with \eqref{kin_mt},
\begin{align}\label{tra_1}
\begin{aligned}
    \tot{X(t;x,v)}{t} &= V(t;x,v),\\
    \tot{V(t;x,v)}{t} &= F[f_{t-\tau}]\lt(Z(t;x,v)\rt),
\end{aligned}
\end{align}
where we adopt the notation $f_{t-\tau}:\equiv g_{t-\tau}$ for $t \in [0,\tau]$.
The system \eqref{tra_1} is considered subject to the initial conditions
\(  \label{tra_1_IC}
   X(0;x,v) = x,\qquad V(0;x,v) = v,
\)
for all $(x,v)\in\R^{2d}$.
Then, by Lemma \ref{lem:LipschF}, we obtain the well-posedness of the characteristic system \eqref{tra_1}--\eqref{tra_1_IC} on the time interval $[0,\tau]$.
Note that the evolution of $V$ can be rewritten as
\[
  \frac{dV(t;x,v)}{dt} = \frac{\int_{\R^{2d}} \psi(|X(t;x,v) - y|) w\,df_{t - \tau}(y,w)}{\int_{\R^{2d}} \psi(|X(t;x,v) - y|)\,df_{t - \tau}(y,w)} - V(t;x,v).
\]
Using a similar argument as in the proof of Lemma \ref{lem_bdd}, we obtain
\[
   \frac{d|V(t)|}{dt} \leq R^{t-\tau}_V - |V(t)| \quad \mbox{for} \quad\mbox{for } t\in[0,\tau].
\]
This, together with the continuity argument from the proof of Lemma \ref{lem_bdd}, yields
\[
   R^t_V < R^0_V \quad \mbox{for} \quad t \geq -\tau \quad\mbox{for } t\in[0,\tau],
\]
which further implies $R^t_X \leq R^0_X + t R^0_V$ for $t\in[0, \tau]$.

Thus, on the time interval $[0, \tau]$ we can construct a solution $f_t$ of \eqref{kin_mt}
that is compactly supported in $x$ and $v$ according to the above estimates.
We then iterate the construction inductively on the time intervals $[k\tau, (k+1)\tau]$ for $k=1,2,\dots$, until we reach the final time $T$.
This provides a solution $f_t\in \mc([0,T];\P_1(\R^{2d}))$ that is uniformly compactly supported in $x$ and $v$ .
Finally, it can be easily proved that this solution can be expressed as a push-forward of the initial datum $f_0$ 
by the characteristic map $Z(t;\cdot,\cdot)$, i.e., $f_t = Z(t;\cdot,\cdot)\# f_0$,
and that this formulation is equivalent to the weak formulation \eqref{weak_for}, see, e.g., \cite{CCR}.
\end{proof}

%
\subsection{Stability estimates for the kinetic equation \eqref{kin_mt}}
In this section we derive the stability estimate for the measure-valued solutions of the system \eqref{kin_mt} constructed in Theorem \ref{main_thm2}.

\begin{theorem}\label{main_thm3}
Let $f^i_t \in \mc([0, T];\P_1(\R^{2d}))$, $i=1,2$, be two weak solutions of \eqref{kin_mt} on the time interval $[0, T]$,
subject to the compactly supported initial data $g^i_s\in \mc([-\tau, 0];\P_1(\R^{2d}))$, as constructed in Theorem \ref{main_thm2}.
Then there exists a constant $C=C(T)$ such that
\(  \label{d-estimate}
   d_1(f^1_t,f^2_t) \leq C \max_{s\in[-\tau,0]} d_1(g^1_s,g^2_s) \quad \mbox{for} \quad t \in [0,T].
\)
\end{theorem}

\begin{proof}
Adopting again the notation $f_{t-\tau}:\equiv g_{t-\tau}$ for $t \in [0,\tau]$,
we again construct the system of characteristics $Z^i(t;x,v) := (X^i(t;x,v),V^i(t;x,v)): [0,T] \times \R^d \times \R^d \to \R^d \times \R^d$,
\begin{align*}
\begin{aligned}
    \tot{X^i(t;x,v)}{t} &= V^i(t;x,v),\\
    \tot{V^i(t;x,v)}{t} &= F[f^i_{t-\tau}]\lt(Z^i(t;x,v)\rt),
\end{aligned}
\end{align*}
subject to the initial condition
\[
   X^i(0;x,v) = x,\qquad V^i(0;x,v) = v,
\]
for all $(x,v)\in\R^{2d}$.
Since by Theorem \ref{main_thm2} the measures $f^i_t$ have
uniformly compact supports in phase space on $[0,T]$,
the flows $Z^i$, $i=1, 2$, are well defined on this time interval.
Then it is easy to check that $f^i_t = Z^i(t;s,\cdot,\cdot) \# f^i_{s}$ for any $t, s \in [0,T]$; see, e.g., \cite{CCR}.
Morever, similarly as in \eqref{supp_diams}, let us denote by $R^T_{i;X}$ and, resp., $R^T_{i;V}$ the support diameters
\[
    R^T_{i;X} := \max_{-\tau \leq s \leq t} R_X[f_i^s],\qquad  R^T_{i;V} := \max_{-\tau \leq s \leq t} R_V[f_i^s],
\]
where $R_X$ and $R_V$ are defined in \eqref{RXRV}.

We choose an \emph{optimal} transport map $\ms^0(x,v) = (\ms_x^0(x,v), \ms_v^0(x,v))$
between the probability measures $f^1_{0}$ and $f^2_{0}$ with respect to the distance $d_1$,
i.e., $f^2_{0} = \ms^s \# f^1_{0}$ and
\[
   d_1(f^1_{0},f^2_{0}) = \int_{\R^{2d}} \left| (x,v) - \ms^0(x,v) \right| \,\d f^1_{0}(x,v).
\]
Moreover, defining $\mt^t := Z^2(t; \cdot, \cdot) \circ \ms^0 \circ Z^1(t;\cdot,\cdot)^{-1}$ for $t \in [0,T]$,
we have $\mt^t \# f^1_t = f^2_t$, and
\[
   d_1(f^1_t,f^2_t) \leq \int_{\R^{2d}} \left|(x,v) - \mt^t(x,v))\right| \,\d f^1_{t}(x,v).
\]
Therefore, defining for $t\in [0,T]$,
$$\begin{aligned}
   u(t) &:= \int_{\R^{2d}} \left|(x,v) - \mt^t(x,v))\right| \,\d f^1_{t}(x,v) \\
                  &= \int_{\R^{2d}} \left|Z^1(t;x,v) - Z^2(t;\ms^0(x,v))\right| \,\d f^1_{0}(x,v),
\end{aligned}$$
where we used the identity $\mt^t\circ Z^1(t;\cdot,\cdot) = Z^2(t;\cdot,\cdot)\circ\ms^0$,
we have $d_1(f^1_t,f^2_t) \leq u(t)$ for all $t\in[0,T]$.
We extend the definition of $\mt^t$ for $t\in[-\tau,0)$ to be an \emph{optimal}
transport map between the data $g^1_{t}$ and $g^2_{t}$. Note that $\mt^0\equiv \ms^0$.
We also extend the definition of $u(t)$,
\[
   u(t) := d_1(g^1_t,g^2_t) = \int_{\R^{2d}} \left| (x,v) - \mt^t(x,v) \right| \,\d g^1_{t}(x,v)
   \qquad\mbox{for } t\in[-\tau,0).
\]
We have
$$\begin{aligned}
   \tot{}{t} u(t) &\leq \int_{\R^{2d}} |V^1(t;x,v) - V^2(t;\ms^0(x,v))|\,\d f^1_{0}(x,v) \\
   &\quad + \int_{\R^{2d}} |F[f^1_{t-\tau}](Z^1(t;x,v)) - F[f^2_{t-\tau}](Z^2(t;\ms^0(x,v)))| \,\d f^1_{0}(x,v).
\end{aligned}$$
The first term of the right-hand side is, by definition, estimated from above by $u(t)$.
The second term is rewritten as
\[
   J := \int_{\R^{2d}} |F[f^1_{t-\tau}](x,v) - F[f^2_{t-\tau}](\mt^t(x,v))| \,\d f^1_{t}(x,v)
\]
and by definition of $F[\cdot]$ we have
$$\begin{aligned}
   & |F[f^1_{t-\tau}](x,v) - F[f^2_{t-\tau}](\mt^t(x,v))| \\
   &\qquad = \left| 
     \frac{\int_{\R^{2d}} \psi(|x-y|)(w-v) \d f^1_{t-\tau}(y,w)}{\int_{\R^{2d}} \psi(|x-y|)\d f^1_{t-\tau}(y,w)}
      - \frac{\int_{\R^{2d}} \psi(|\mt^t_x-y|)(w-\mt^t_v) \d f^2_{t-\tau}(y,w)}{\int_{\R^{2d}} \psi(|\mt^t_x-y|)\d f^2_{t-\tau}(y,w)}
      \right|,
\end{aligned}$$
where we introduced the shorthand notation $\mt^t_x:=\mt^t_x(x,v)$ and similarly for $\mt^t_v$.
Note that due to the monotonicity of $\psi$, we have the lower bound for the first denominator
\[
   \int_{\R^{2d}} \psi(|x-y|)\d f^1_{t-\tau}(y,w) \geq \psi(R^T_{1;X}) > 0 \quad\mbox{for all } x\in\R^d,
\]
and analogously for the other denominator.
Consequently,
$$\begin{aligned}
   & |F[f^1_{t-\tau}](x,v) - F[f^2_{t-\tau}](\mt^t(x,v))| \\
   &\qquad \leq \frac{1}{\psi(R^T_{1;X})} \left| \int_{\R^{2d}} \psi(|x-y|)(w-v) \d f^1_{t-\tau}(y,w) - \int_{\R^{2d}} \psi(|\mt^t_x-y|)(w-\mt^t_v) \d f^2_{t-\tau}(y,w) \right|  \\
   &\qquad\qquad + \frac{1}{\psi(R^T_{1;X}) \psi(R^T_{2;X})} \left| \int_{\R^{2d}} \psi(|\mt^t_x-y|)(w-\mt^t_v) \d f^2_{t-\tau}(y,w) \right| \\
   &\qquad\qquad\qquad \times  \left| \int_{\R^{2d}} \psi(|x-y|) \d f^1_{t-\tau}(y,w) - \int_{\R^{2d}} \psi(|\mt^t_x-y|) \d f^2_{t-\tau}(y,w) \right|.
\end{aligned}$$
The expression on the last line is estimated by
$$\begin{aligned}
   & \left| \int_{\R^{2d}} \psi(|x-y|) \d f^1_{t-\tau}(y,w) - \int_{\R^{2d}} \psi(|\mt^t_x-y|) \d f^2_{t-\tau}(y,w) \right| \\
     &\qquad \leq \int_{\R^{2d}} \left| \psi(|x-y|) - \psi(|\mt^t_x-\mt^{t-\tau}_x(y,w)|)  \right|  \d f^1_{t-\tau}(y,w)  \\
     &\qquad \leq \|\psi\|_{Lip} \left( \left| x - \mt^t_x \right| + \int_{\R^{2d}} \left| y - \mt^{t-\tau}_x(y,w) \right|  \d f^1_{t-\tau}(y,w) \right) \\
     &\qquad \leq \|\psi\|_{Lip} \left( \left| x - \mt^t_x \right| + u(t-\tau) \right).     
\end{aligned}$$
Similarly, we have
$$\begin{aligned}
   & \left| \int_{\R^{2d}} \psi(|x-y|)(w-v) \d f^1_{t-\tau}(y,w) - \int_{\R^{2d}} \psi(|\mt^t_x-y|)(w-\mt^t_v) \d f^2_{t-\tau}(y,w) \right| \\
     &\qquad \leq \int_{\R^{2d}} \left| \psi(|x-y|)(w-v) - \psi(|\mt^t_x-\mt^{t-\tau}_x(y,w)|)(\mt^{t-\tau}_v(y,w)-\mt^t_v)  \right|  \d f^1_{t-\tau}(y,w)\\
     &\qquad \leq \int_{\R^{2d}} \left| \psi(|x-y|) - \psi(|\mt^t_x-\mt^{t-\tau}_x(y,w)|) \right| |v-w|  \d f^1_{t-\tau}(y,w)\\
     &\qquad\qquad + \int_{\R^{2d}} \psi(|\mt^t_x-\mt^{t-\tau}_x(y,w)|)  \left|v-w - (\mt^{t-\tau}_v(y,w)-\mt^t_v)\right|  \d f^1_{t-\tau}(y,w).
\end{aligned}$$
The first term of the right-hand side is estimated by
$$\begin{aligned}
&\int_{\R^{2d}} \left| \psi(|x-y|) - \psi(|\mt^t_x-\mt^{t-\tau}_x(y,w)|) \right| |v-w|  \d f^1_{t-\tau}(y,w)\cr
&\qquad    \leq (|v| + R_v^1) \|\psi\|_{Lip} \bigl( \left| x - \mt^t_x \right| + u(t-\tau) \bigr),
\end{aligned}$$
where we used the uniform boundedness of the velocity support of $f^1_t$ and similar steps as above.
For the second term we have
$$\begin{aligned}
     &\int_{\R^{2d}} \psi(|\mt^t_x-\mt^{t-\tau}_x(y,w)|) \left|v-w - (\mt^{t-\tau}_v(y,w)-\mt^t_v)\right|  \d f^1_{t-\tau}(y,w) \\
     &\qquad \leq \|\psi\|_{L^\infty} \int_{\R^{2d}} \left( |v-\mt^t_v| + |w - \mt^{t-\tau}_v(y,w)| \right)  \d f^1_{t-\tau}(y,w) \\
     &\qquad \leq \|\psi\|_{L^\infty} \bigl(  |v-\mt^t_v| + u(t-\tau)  \bigr).
\end{aligned}$$
Finally, we have
\[
   \left| \int_{\R^{2d}} \psi(|\mt^t_x-y|)(w-\mt^t_v) \d f^2_{t-\tau}(y,w) \right|
     \leq \|\psi\|_{L^\infty}( R^T_{2;V} + |\mt^t_v| ).
\]
Putting the above estimates together, we arrive at
\[
   J &=& \int_{\R^{2d}} |F[f^1_{t-\tau}](x,v) - F[f^2_{t-\tau}](\mt^t(x,v)| \,\d f^1_{t}(x,v) \\
     &\leq& C \int_{\R^{2d}} ( |x - \mt^t_x| + |v-\mt^t_v| + u(t-\tau) ) \,\d f^1_{t}(x,v) \\
     &\leq& C ( u(t) + u(t-\tau) ),
\]
where the constant $C$ depends only on $\|\psi\|_{L^\infty}$, $\|\psi\|_{Lip}$
and the support diameters $R^T_{i;X}$, $R^T_{i;V}$, $i=1,2$.

Finally, we arrive at
\[
   \tot{}{t} u(t) \leq C ( u(t) + u(t-\tau) )
\]
for all $t \in [0,T]$.
To conclude \eqref{d-estimate}, we denote $$\bar u:=\max_{s\in[-\tau,0]} u(s) = \max_{s\in[-\tau,0]} d_1(g^1_s,g^2_s),$$
and for $w(t):=e^{-Ct} u(t)$, we calculate
\[
   \tot{}{t} w(t) \leq C e^{-C\tau} w(t-\tau).
\]
A simple induction argument yields then
\[
    w(t) \leq \bar u \left(1+C\tau e^{-C\tau}\right)^k \qquad \mbox{for } t\in((k-1)\tau, k\tau].
\]
Consequently,
\[
    u(t) \leq \bar u e^{Ct} \left(1+C\tau e^{-C\tau}\right)^k \qquad \mbox{for } t\in((k-1)\tau, k\tau],
\]
which can be further roughly estimated by
\[
    u(t) \leq \bar u e^{2Ct}\qquad\mbox{for } t\geq 0.
\]
We conclude \eqref{d-estimate} by recalling that $d_1(f^1_t,f^2_t) \leq u(t)$ for all $t\in[0,T]$.
\end{proof}

\begin{remark}\label{rem:mean_field}
The stability result of Theorem \ref{main_thm3} can be used for carrying out a rigorous passage
to the mean field limit in the discrete system \eqref{main_eq}--\eqref{IC0}.
Indeed, let us fix an initial datum $g_s\in \mc([-\tau, 0];\P_1(\R^{2d}))$,
compactly supported in position and velocity, i.e., satisfying \eqref{kinetic_supp} with some $R>0$.
Let $\{g^N_s\}_{N\in\N}$ be a family $N$-particle approximations of $g_s$, i.e.,
\[
   g^N_s := \sum_{i=1}^N  \delta(x-x^0_i(s))\otimes \delta(v-v^0_i(s)) \qquad\mbox{for } s\in[-\tau,0],
\]
where the $x_i^0, v_i^0 \in \mc([-\tau,0];\R^d)$ are chosen such that
\[
   \max_{s\in[-\tau,0]} d_1(g^N_s,g_s) \to 0 \quad\mbox{as}\quad N\to\infty.
\]
Denoting then $(x^N_i,v^N_i)$ the solution of the discrete $N$-particle system \eqref{main_eq}--\eqref{IC0}
subject to the initial datum $(x_i^0,v_i^0)_{i=1,\dots,N}$, and the corresponding empirical measure
\[
   f^N_t := \sum_{i=1}^N  \delta(x-x^N_i(t))\otimes \delta(v-v^N_i(t)) \qquad\mbox{for } t\in[0,T),
\]
then it is easily checked that $f^N_t$ is a measure valued solution of the kinetic system \eqref{kin_mt}
in the sense of Definition \ref{def_weak}.
Moreover, if $f_t\in \mc([0,T);\P_1(\R^{2d}))$ is a measure valued solution of \eqref{kin_mt}
subject to the datum $g_s$, as constructed in Theorem \ref{main_thm2},
then by Theorem \ref{main_thm3} we have the stability estimate
\[
   d_1(f_t,f^N_t) \leq C \max_{t\in[-\tau,0]} d_1(g_s,g^N_s) \quad \mbox{for} \quad t \in [0,T),
\]
where the constant $C$ depends only on the influence function $\psi$, $R$ and $T$
(in particular, it is independent of $N$).
Consequently, $f^N_t$ is an approximation of $f_t$, i.e., $f^N_t \to f_t$ in $d_1$,
uniformly on $[0,T)$, as $N\to\infty$.
\end{remark}

%
%
%
%
\subsection{Asymptotic flocking in the kinetic equation \eqref{kin_mt}}
In this part we present a sufficient condition for asymptotic flocking in the kinetic system \eqref{kin_mt}.
Let us note that this is a very natural extension of Theorem \ref{thm_main},
combined with the stability result of Theorem \ref{main_thm3}.

In analogy to \eqref{dXdV}, we define the position- and velocity diameters for a compactly supported measure $g\in\P_1(\R^{2d})$,
\[
   d_X[g] := \mbox{diam}\lt({\mbox{supp}_x g}\rt), \qquad  d_V[g] := \mbox{diam}\lt({\mbox{supp}_v g}\rt),
\]
where supp$_x f$ denotes the $x$-projection of supp$f$ and similarly for supp$_v f$.

\begin{theorem}\label{main_thm4}
Let $f_t \in \mc([0, T);\P_1(\R^{2d}))$ be a weak solution of \eqref{kin_mt} on the time interval $[0, T)$,
subject to a compactly supported initial datum $g_s\in \mc([-\tau, 0];\P_1(\R^{2d}))$, as constructed in Theorem \ref{main_thm2}.
Moreover, assume that
\( \label{kin_flocking_cond}
   d_V[g_0]  + \int_{-\tau}^0 d_V[g_s]\,\d s < \int_{d_X[g_{-\tau}] + R_V^0 \tau}^\infty \psi(s)\,\d s.
\)
Then the weak solution $f_t$ satisfies
\(  \label{kin_flocking}
   d_V[f_t] \leq \left( \max_{s\in[-\tau,0]} d_V[g_s] \right) e^{-Ct} \quad \mbox{for } t \geq 0,
   \qquad \sup_{t\geq 0} d_X[f_t] < \infty,
\)
where $C$ is a positive constant independent of $t$.
\end{theorem}

\begin{proof}
Similarly as in Remark \ref{rem:mean_field}, we construct
$\{g^N_s\}_{N\in\N}$ a family of $N$-particle approximations of $g_s$, i.e.,
\[
   g^N_s := \sum_{i=1}^N  \delta(x-x^0_i(s))\otimes \delta(v-v^0_i(s)) \qquad\mbox{for } s\in[-\tau,0],
\]
where the $x_i^0, v_i^0 \in \mc([-\tau,0];\R^d)$ are chosen such that
\[
   \max_{s\in[-\tau,0]} d_1(g^N_s,g_s) \to 0 \quad\mbox{as}\quad N\to\infty.
\]
Due to the assumption \eqref{kin_flocking_cond}, we can choose $x_i^0, v_i^0$
such that the discrete flocking condition \eqref{ass1} is uniformly satisfied for all $N\in\N$.
Denoting then $(x^N_i,v^N_i)$ the solution of the discrete $N$-particle system \eqref{main_eq}--\eqref{IC0}
subject to the initial datum $(x_i^0,v_i^0)_{i=1,\dots,N}$, Theorem \ref{thm_main}
provides a positive constant $C_1>0$ such that
\[
   d_V(t) \leq \left( \max_{s\in[-\tau,0]} d_V(s) \right) e^{-C_1 t} \quad \mbox{for } t \geq 0,
\]
with the diameters $d_V$, $d_X$ defined in \eqref{dXdV}.
The constant $C_1>0$ is independent of $t$ and $N$.
The empirical measure
\[
   f^N_t := \sum_{i=1}^N  \delta(x-x^N_i(t))\otimes \delta(v-v^N_i(t))
\]
is a measure valued solution of the kinetic equation \eqref{kin_mt}
in the sense of Definition \ref{def_weak}.
For any fixed $T>0$, Theorem \ref{main_thm3} provides the stability estimate
\[
   d_1(f_t,f^N_t) \leq C_2 \max_{s\in[-\tau,0]} d_1(g_s,g^N_s) \quad \mbox{for} \quad t \in [0,T),
\]
where the constant $C_2>0$ is independent of $N$.
Thus, fixing $T>0$ and letting $N\to\infty$ implies $d_V[f_t] = d_V(t)$ on $[0,T)$, and, consequently,
\[
   d_V[f_t] \leq \left( \max_{s\in[-\tau,0]} d_V[g_s] \right) e^{-C_1 t} \quad \mbox{for} \quad t \in[0,T).
\]
Since $T$ can be chosen arbitrarily and $C_1$ is independent of time, we conclude \eqref{kin_flocking}.
The finiteness of $\sup_{t\geq 0} d_X[f_t]$ is a direct consequence of the above.
\end{proof}

%
%
%
%
\section{Numerical experiments}\label{sec_nur}
In this section we present several numerical experiments for the particle system \eqref{main_eq}--\eqref{IC0}.
We use the standard explicit Euler scheme for the discretization in time and the method of steps
to treat the delayed terms, see, e.g., \cite{Smith}. We are interested in possible oscillatory behavior
of the particle velocities and their long-time behavior.

\subsection{Two particles}
We first consider the case of two particles, $N=2$, with positions $x_1(t), x_2(t)$ and velocities $v_1(t), v_2(t)$.
Then $\phi_{12} = \phi_{21} = 1$, thus the velocity equations in \eqref{main_eq} decouple from the positions and they satisfy
\begin{align}\label{2part}
\begin{aligned}
\frac{d v_1(t)}{dt} &= v_2(t - \tau) - v_1(t), \\
\frac{d v_2(t)}{dt} &= v_1(t - \tau) - v_2(t).
\end{aligned}
\end{align}
Defining $u:=v_1+v_2$, $w:=v_1-v_2$, we have
\(
  \frac{d u(t)}{dt} &=& u(t - \tau) - u(t), \label{uEq} \\
  \frac{d w(t)}{dt} &=& -w(t - \tau) - w(t). \label{wEq}
\)
The first equation admits, for a constant initial datum, a constant solution,
which corresponds to momentum conservation of the two-particle system.
Assuming a solution of the form $w(t) = e^{\lambda t}$ for the second equation, for some $\lambda\in\mathbb{C}$,
we obtain the characteristic equation
\[
   \lambda = -e^{-\lambda\tau} - 1.
\]
Writing $\lambda = \mu + i\sigma$ with $\mu$, $\sigma\in\R$, the real part
of the characteristic equation reads
\[
   \mu + 1 = - e^{-\mu\tau} \cos(\sigma\tau).
\]
This immediately implies that $\mu \leq 0$, and from the asymptotic stability theory
of delayed differential equations, see, e.g. \cite{Smith, Halanay}, it follows
that all solutions $w(t)$ of \eqref{wEq} either tend to zero as $t\to\infty$
or stay uniformly bounded.
Using the Lyapunov function method developed in \cite{EHS}, it can be shown that
$w$ asymptotically converges to zero as $t\to\infty$ whenever $\tau < 1/\sqrt{2}$;
however, this sufficient condition seems not to be optimal.
Moreover, the numerical experiments presented in Fig. \ref{fig:1}
suggest that for small delays ($\tau=0.25$, first row in Fig. \ref{fig:1}) the solution typically tends to zero
monotonically as $t\to\infty$, while oscillations appear 
for larger delays ($\tau=1$, second row in Fig. \ref{fig:1})

\begin{figure}[ht] \centering
 \includegraphics[width=.49\textwidth]{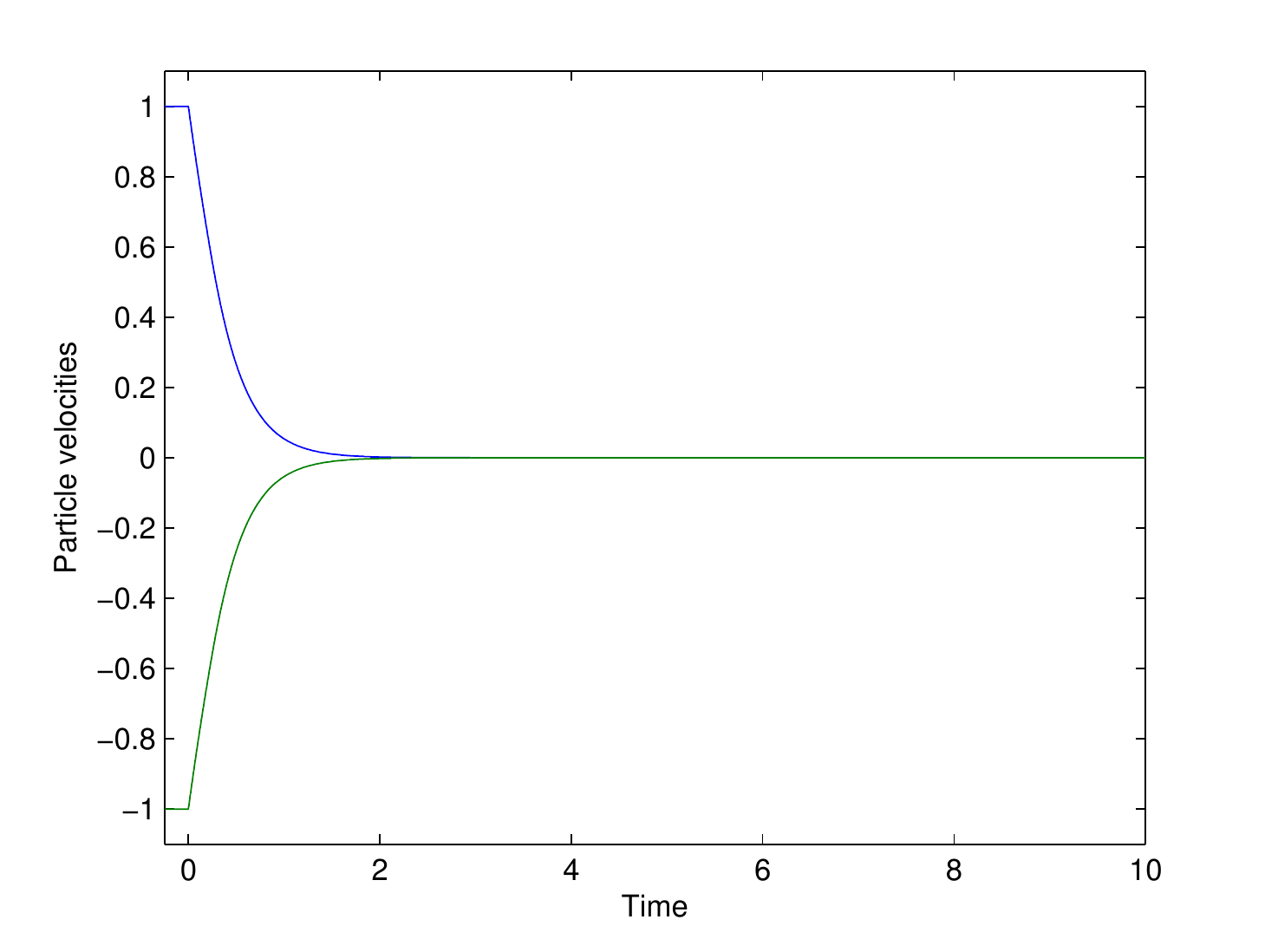}
 \includegraphics[width=.49\textwidth]{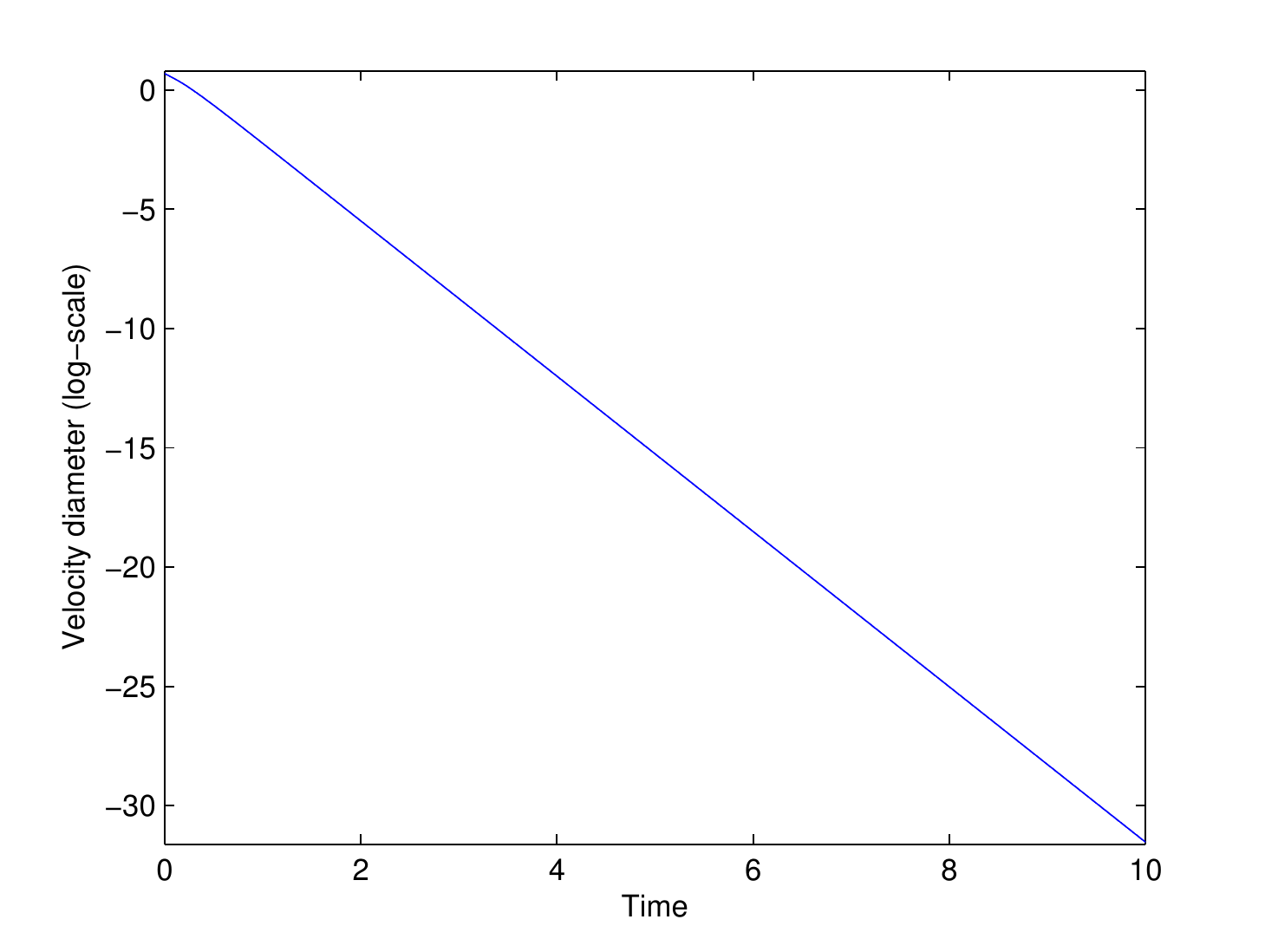}\\
 \includegraphics[width=.49\textwidth]{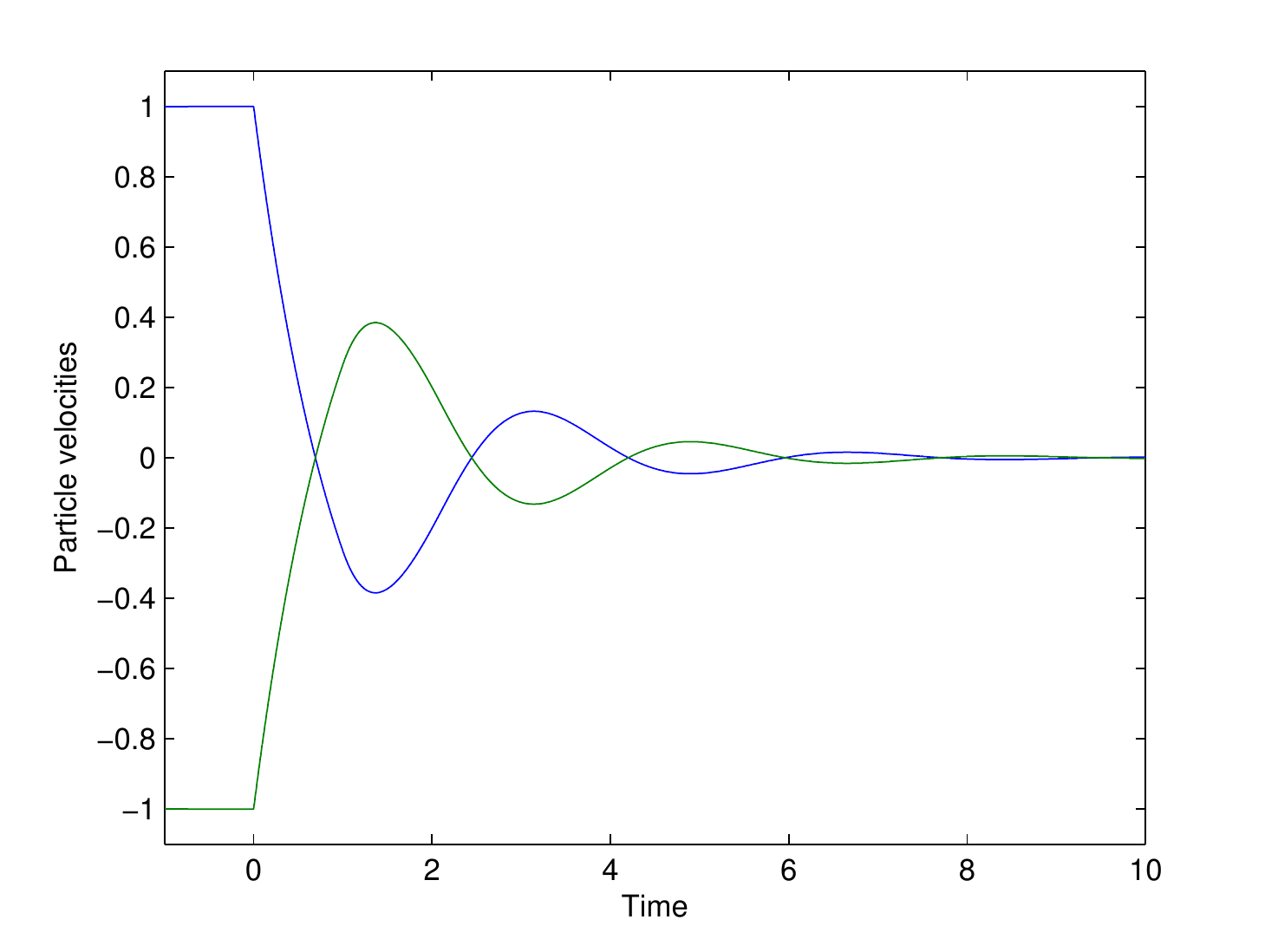}
  \includegraphics[width=.49\textwidth]{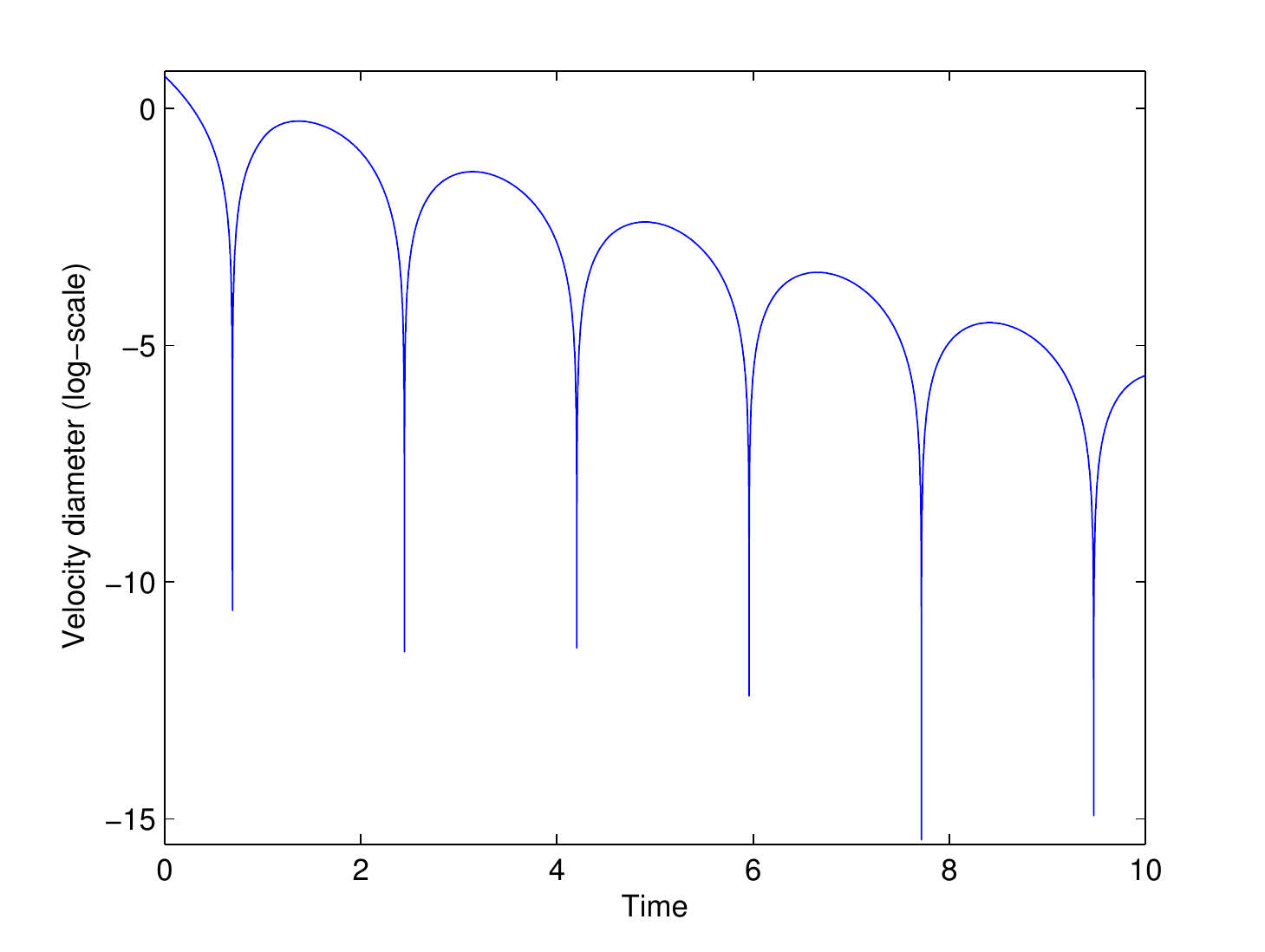}\\
  \caption{The system with two particles: Particle velocities $v_1(t)$, $v_2(t)$ as solututions of \eqref{2part}
  (left panels) and velocity diameters $d_V(t)$ (right panels, logarithmic scale)
  for $\tau=0.25$ (first row) and $\tau=1$ (second row). 
  The initial condition is constant, $v_1(t)\equiv 1$, $v_2(t)\equiv -1$ for $t\in[-\tau,0]$.}
  \label{fig:1}
\end{figure}

\subsection{Three particles}
We provide a heuristic argument suggesting that for $N=3$ and (at least) $\tau=0$,
we can expect asymptotic flocking for any initial datum
and any nonincreasing positive influence function $\psi$. 
Indeed, for any permutation $(i,j,k)$ of the particle indices $\{1,2,3\}$
we have the following alternative: Either the distance $|x_i-x_j|$
is larger than $|x_i-x_k|$, then $\psi(|x_i-x_j|) \leq \psi(|x_i-x_k|)$ and,
consequently,
\[
   \phi_{ik}(x,0) = \frac{\psi(|x_i-x_k|)}{\psi(|x_i-x_k|) + \psi(|x_i-x_j|)} \geq \frac12.
\]
Or the opposite is true, then $\psi(|x_i-x_j|) \geq \psi(|x_i-x_k|)$ and
\[
   \phi_{ij}(x,0) = \frac{\psi(|x_i-x_j|)}{\psi(|x_i-x_k|) + \psi(|x_i-x_j|)} \geq \frac12.
\]
Repeating this argument for all permutations of the indices $\{1,2,3\}$,
we conclude that each particle interacts "strongly" with at least
one other particle, independently of the particle distances
(even if the influence function $\psi(s)$ is decaying quickly for large $s$).
Thus, having only three particles, we expect asymptotic flocking for any initial datum.

Even though the above argument is heuristic and assumes $\tau=0$,
our extensive numerical simulations seem to suggest that it applies
for any delay length $\tau$. Indeed, we were not able to find a setting where
the velocity fluctuation would not be asymptotically tending to zero.
We illustrate this in Fig. \ref{fig:2}, where we solved the system \eqref{main_eq}--\eqref{IC0}
with exponentially decaying influence function $\psi(s) = e^{-s}$,
subject to the initial datum
\( \label{IC_N3_1}
   v_1(s) \equiv -10,\quad v_2(s) \equiv 0,\quad v_3(s)\equiv 20,\qquad s\in[-\tau, 0]
\)
and
\(  \label{IC_N3_2}
    x_i(s) = v_i s,\qquad s\in[-\tau, 0],
\)
i.e., the initial velocities are constant, and the particle trajectories
start from zero at $t=-\tau$.
We choose three different values for the delay $\tau\in\{10^{-2}, 10^{-1}, 1\}$.
Note that at $t=0$, the particle locations are
\[
   x_1(0)=-10\tau,\quad x_2(0) = 0,\quad x_3(0) = 20\tau,
\]
so for $\tau=1$ the terms $\psi(|x_i-x_k|)$ are exponentially small (in particular,
of the orders $e^{-10}$ and $e^{-20}$).
Still, the solution converges relatively fast to a common velocity (see Fig. \ref{fig:2}, last row).
Moreover, note that for $\tau = 0.25$ the decay
of the velocity diameter is monotone, while for $\tau=1$ oscillations appear.

\begin{figure}[ht] \centering
 \includegraphics[width=.49\textwidth]{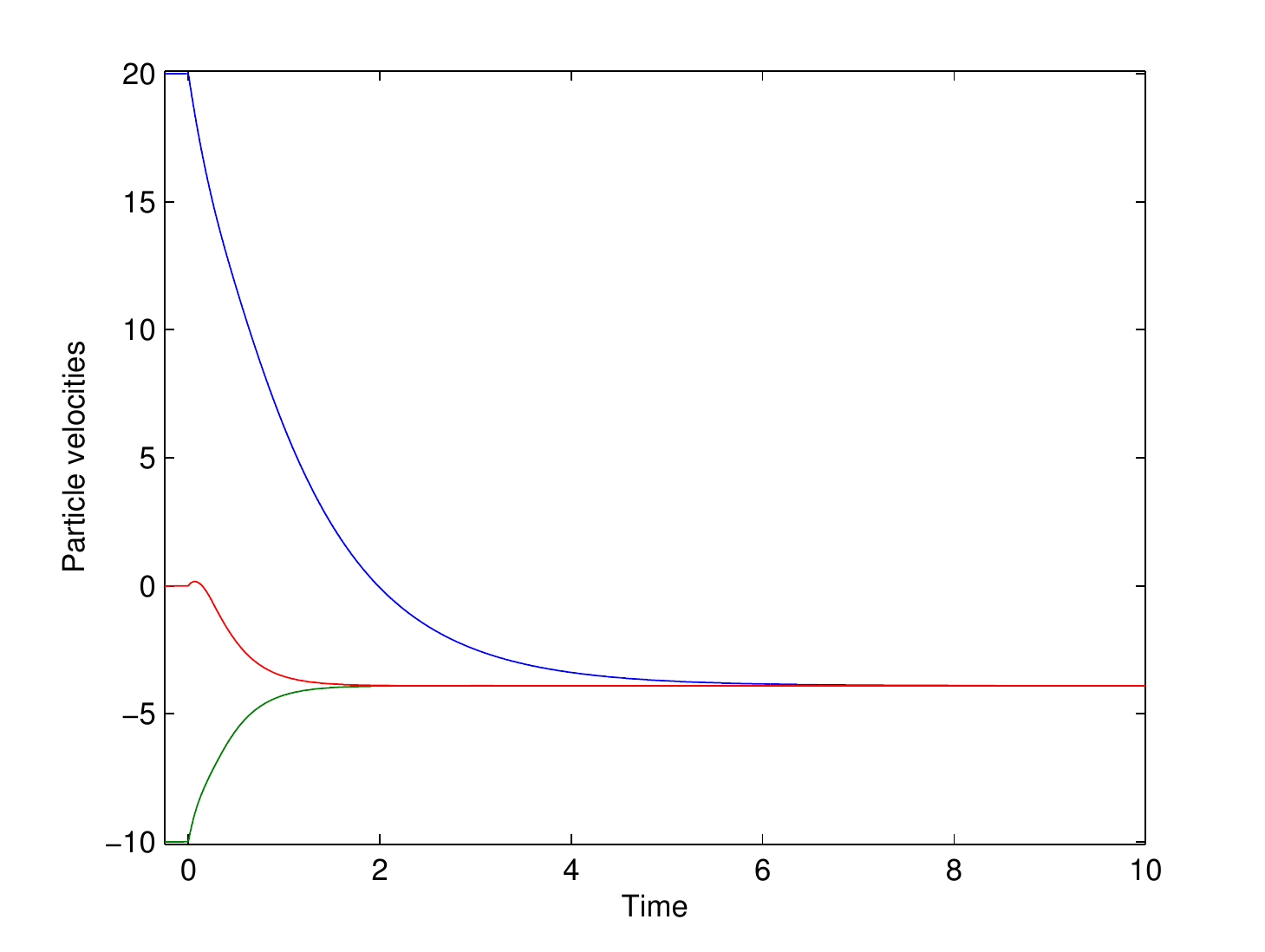}
 \includegraphics[width=.49\textwidth]{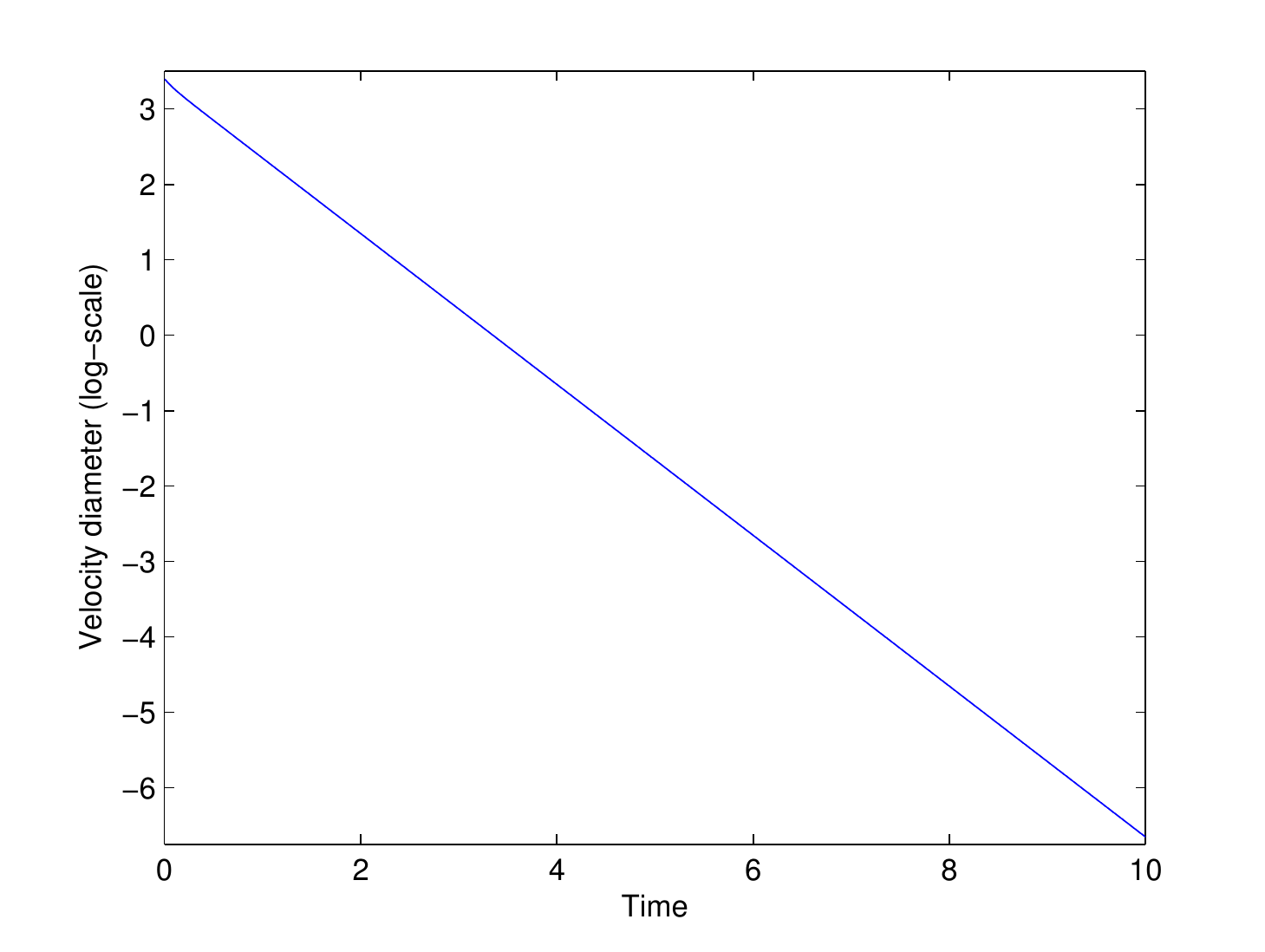}\\
 \includegraphics[width=.49\textwidth]{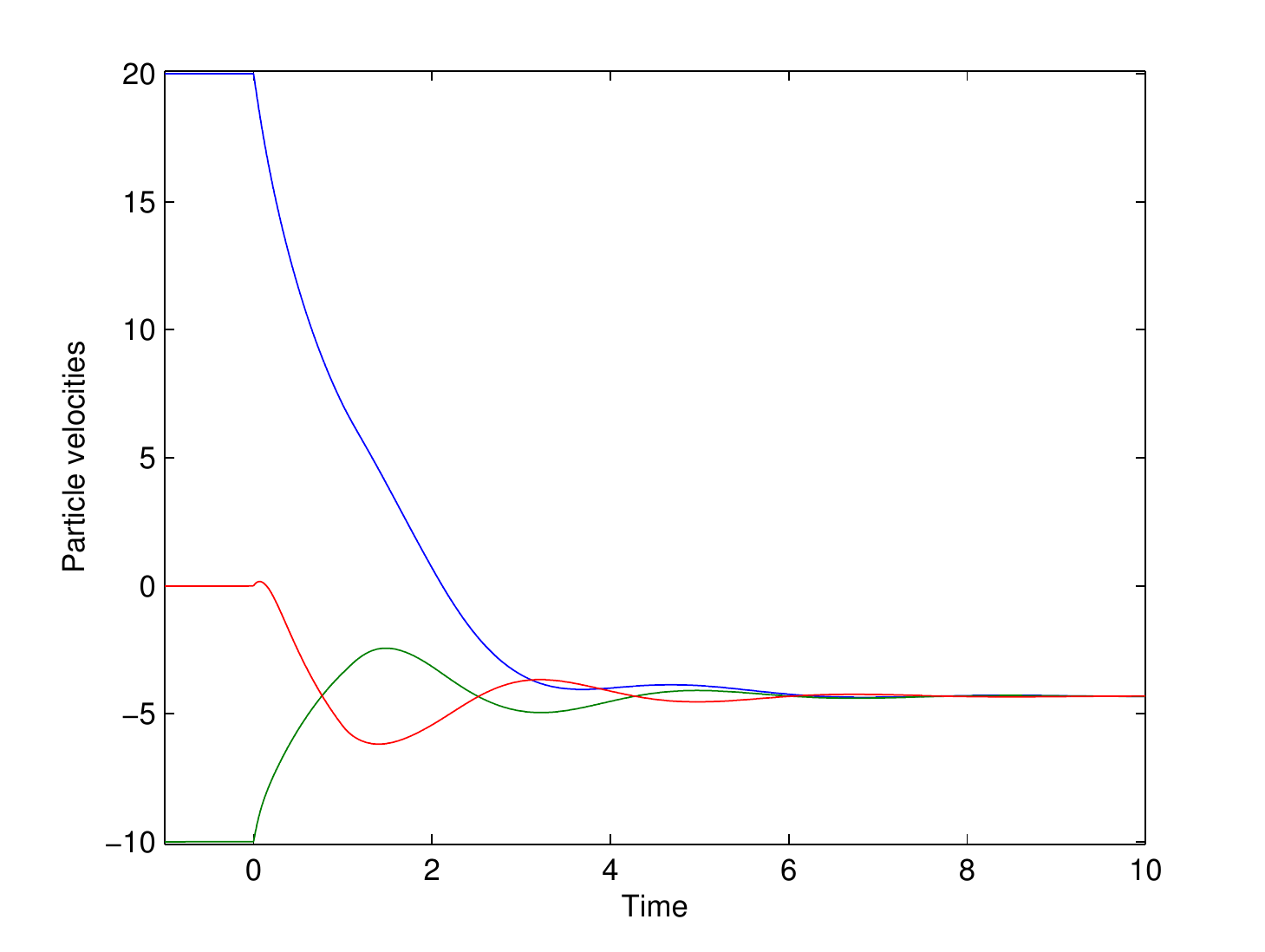}
  \includegraphics[width=.49\textwidth]{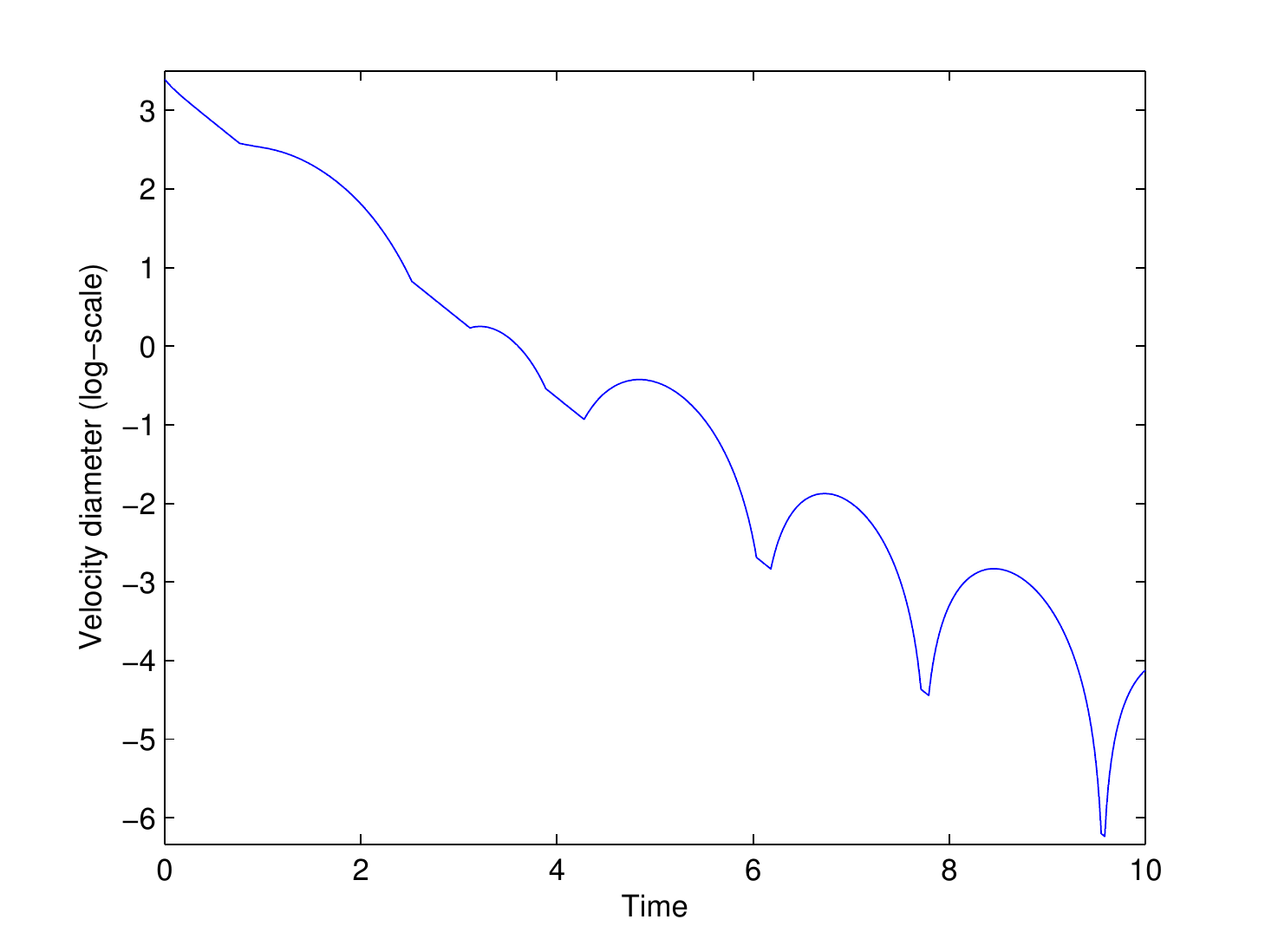}\\
  \caption{The system with three particles: particle velocities $v_1(t)$, $v_2(t), v_3(t)$ as solutions of \eqref{main_eq}--\eqref{IC0}
  (left panels) and velocity diameters $d_V(t)$ (right panels, logarithmic scale)
  for $\tau=0.25$ (first row) and $\tau=1$ (second row), 
  with exponentially decaying influence function $\psi(s) = e^{-s}$.
  The initial condition is in both cases given by \eqref{IC_N3_1}--\eqref{IC_N3_2}.
}
  \label{fig:2}
\end{figure}

\subsection{Four particles}
According to the heuristic argument above and our numerical experiments,
it seems that at least four particles are needed in order to observe non-flocking
in the system. We present a setting where flocking takes place for a small value of the delay,
but there is no flocking for a larger delay.

We solve the coupled system \eqref{main_eq}--\eqref{IC0} with $N=4$ and the influence function $\psi$
given by the Cucker-Smale-type expression
\[
   \psi(s) = \frac{1}{(1+s^2)^4}.
\]
We prescribe the initial datum
\( \label{IC_N4_1}
   v_1(s) \equiv -0.1,\quad v_2(s) \equiv 0,\quad v_3(s)\equiv 0.5,\quad v_4(s)\equiv 0.6, \qquad s\in[-\tau, 0]
\)
and
\(  \label{IC_N4_2}
    x_i(s) = v_i s,\qquad s\in[-\tau, 0].
\)
We consider two values for the delay, $\tau=0.25$ and $\tau=1$,
and plot the solutions in Fig. \ref{fig:3}.
Let us note that in neither case the flocking condition \eqref{ass1}
of Theorem \ref{thm_main} is satisfied. However, for $\tau=0.25$
the system still exhibits asymptotic flocking. On the other hand,
the reason why there is no flocking for $\tau=1$ is that there is strong interaction
between particles $1$ and $2$ and between particles $3$ and $4$,
but the interaction between those two pairs is weak.

\begin{figure}[ht] \centering
 \includegraphics[width=.49\textwidth]{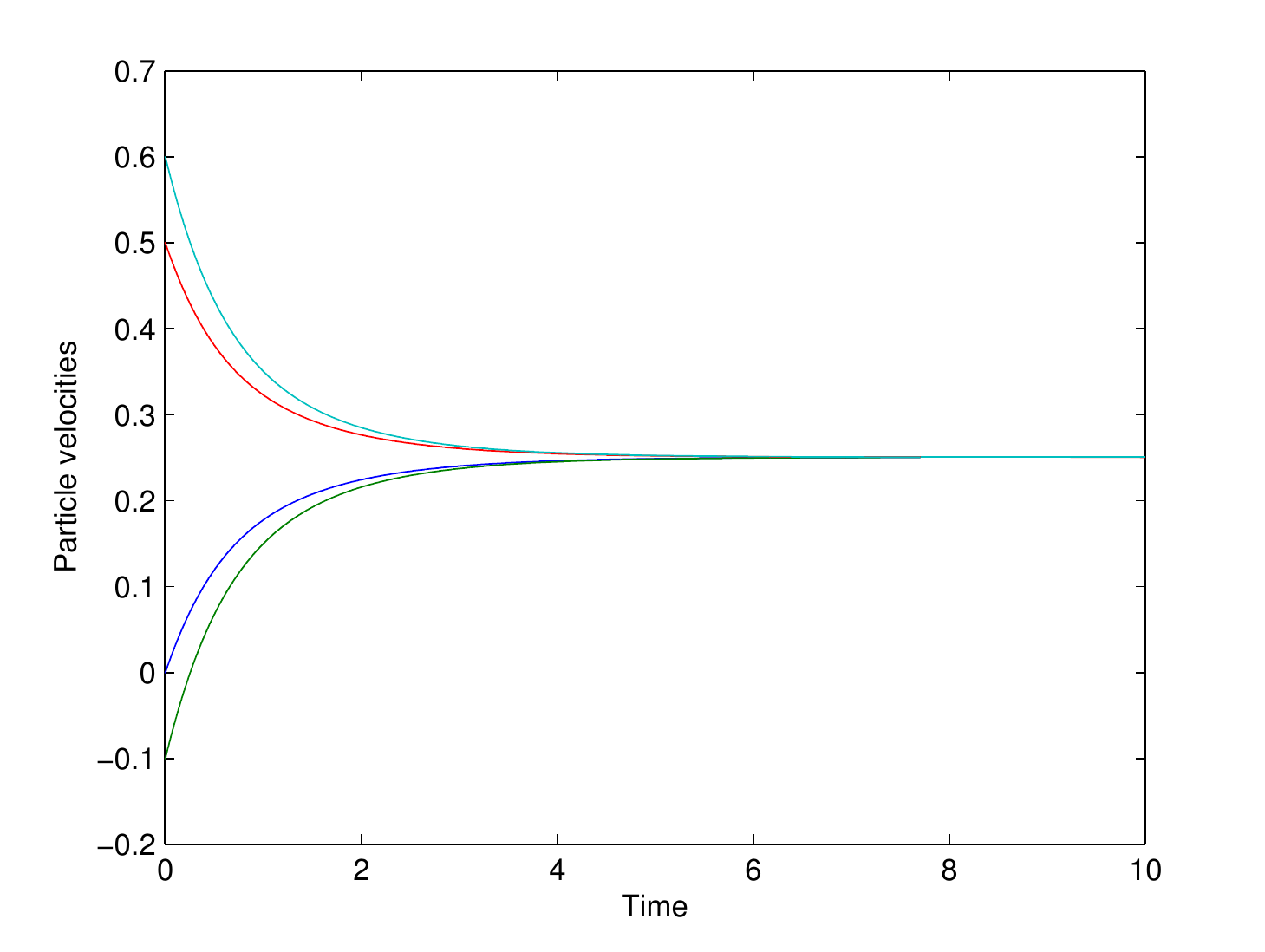}
 \includegraphics[width=.49\textwidth]{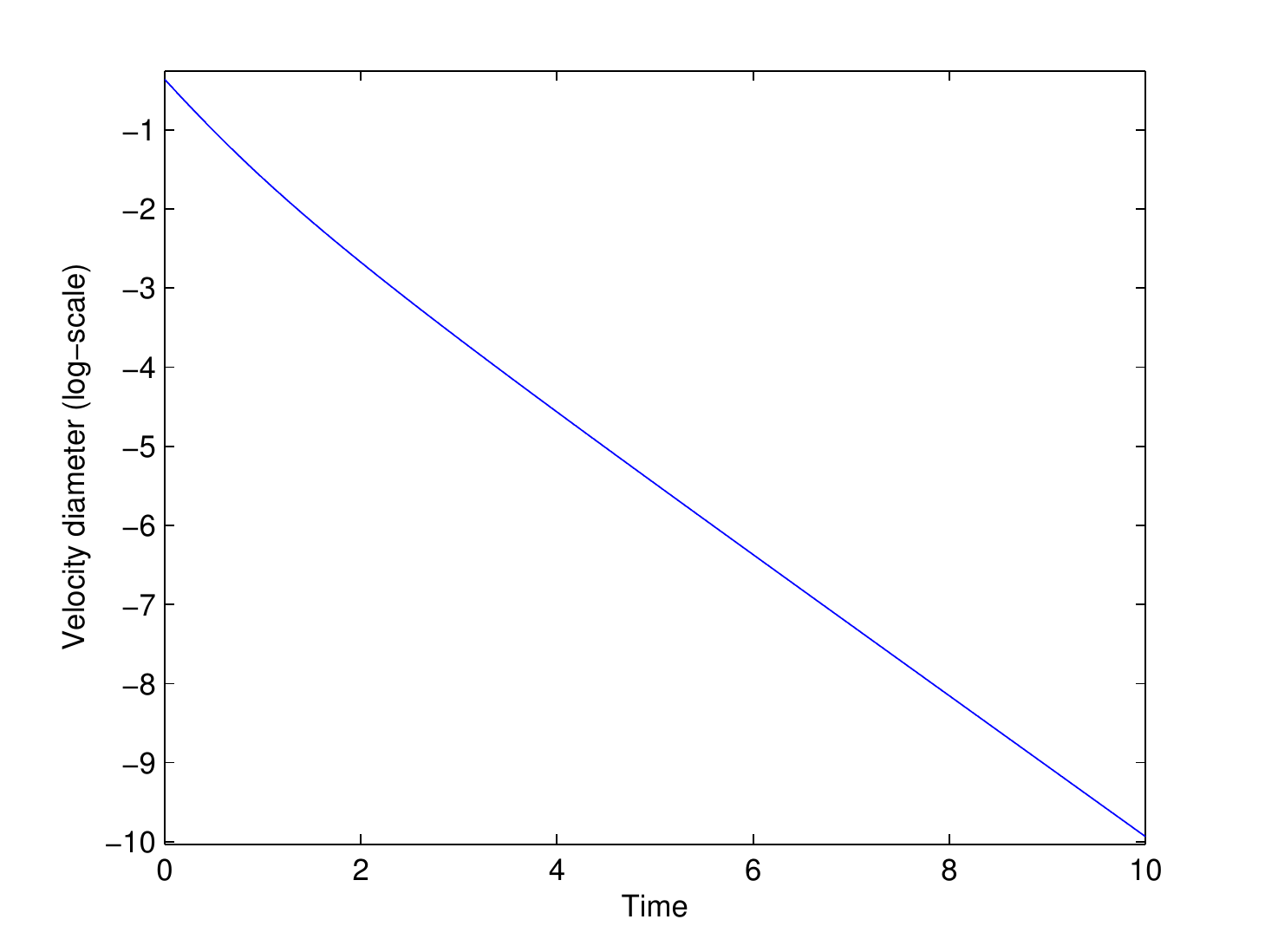}\\
 \includegraphics[width=.49\textwidth]{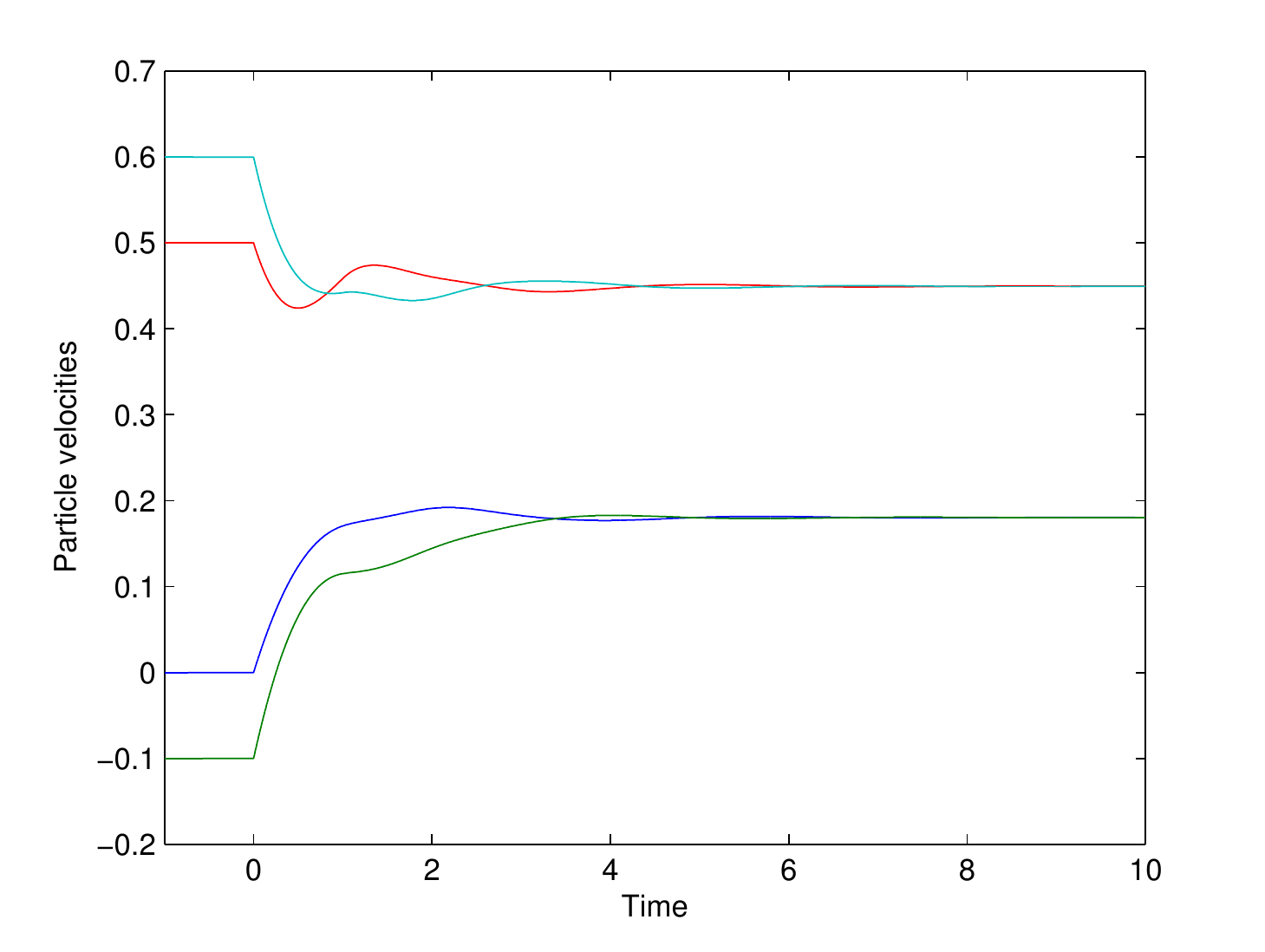}
  \includegraphics[width=.49\textwidth]{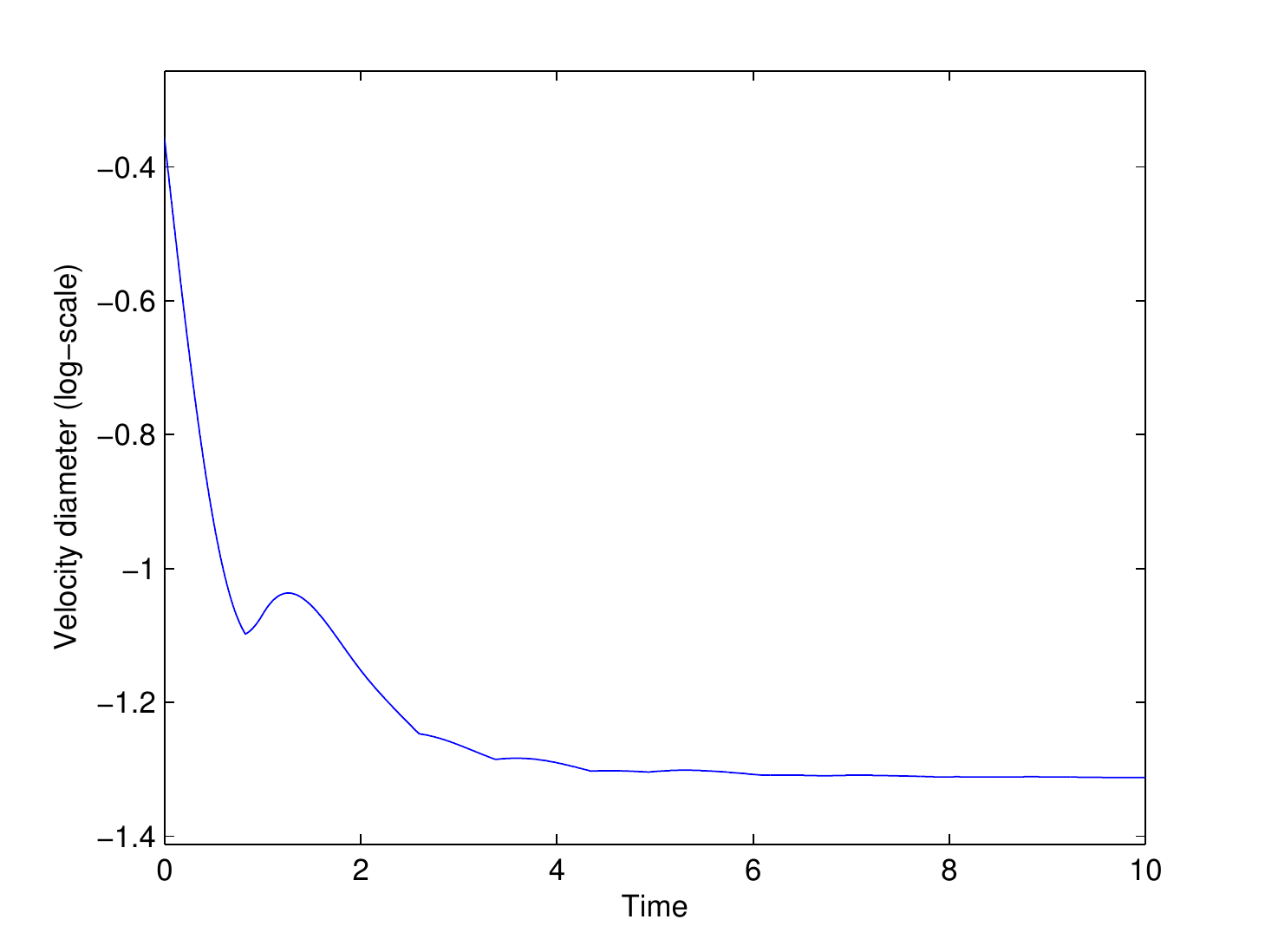}\\
  \caption{The system with four particles: particle velocities as solutions of \eqref{main_eq}--\eqref{IC0}
  (left panels) and velocity diameters $d_V(t)$ (right panels, logarithmic scale)
  for $\tau=0.25$ (first row) and $\tau=1$ (second row), 
  with the influence function $\psi(s) = {(1+s^2)^{-4}}$.
  The initial condition is in both cases given by \eqref{IC_N4_1}--\eqref{IC_N4_2}.
}
  \label{fig:3}
\end{figure}

%
%
%
%

\section*{Acknowledgments}
YPC was supported by Engineering and Physical Sciences Research Council (EP/K00804/1) and ERC-Starting grant HDSPCONTR ``High-Dimensional Sparse Optimal Control''. YPC is also supported by the Alexander Humboldt Foundation through the Humboldt Research Fellowship for Postdoctoral Researchers.
JH was supported by KAUST baseline funds and KAUST grant no. 1000000193.

%
%
%
%

\end{document}